\documentclass[reqno,a4paper,11pt]{amsart}

\usepackage[utf8]{inputenc}
\usepackage{graphicx}
\usepackage{subfigure}
\usepackage{amsmath, amssymb, amsthm, mathtools,times,comment}
\usepackage[foot]{amsaddr}
\usepackage{enumitem}
\usepackage{mathrsfs}
\usepackage[english]{babel}
\usepackage[T1]{fontenc}
\usepackage{fourier}
\usepackage{fullpage}
\usepackage{xparse}
\usepackage{etoolbox} 
\usepackage{hyperref}
\usepackage{cleveref}
\usepackage{crossreftools}

\usepackage{pgfplots}
 \pgfplotsset{compat = newest} 
 \usetikzlibrary{math}
 \usetikzlibrary{calc}
 \usetikzlibrary{shapes.geometric}
 \usetikzlibrary{arrows}

\newcommand{\N}{\mathbb{N}} 
\newcommand{\R}{\mathbb{R}} 
\newcommand{\whp}{whp}

\newcommand{\prob}[1]{\mathbb{P}\left[#1\right]} 

\newcommand{\expec}[1]{\mathbb{E}\left[#1\right]} 

\newcommand{\vNhd}[2][r]{
\ifstrequal{#1}{1}
{N_V \left(#2 \right)}
{N_V^{#1}\left(#2 \right)}
}   

\newcommand{\eNhd}[2][r]{
\ifstrequal{#1}{1}
{N_E \left(#2 \right)}
{N_E^{#1}\left(#2 \right)}
}   

\newcommand{\dist}[2]{d_G (#1,#2 ) } 
\newcommand{\vine}[1]{V_{#1}} 
\newcommand{\cop}[1]{c\left(#1\right)}

\newcommand{\smoothconst}[0]{\xi} 
\newcommand{\smoothexp}[0]{$\smoothconst$-expanding } 
\newcommand{\dhat}{\hat{d}} 
\newcommand{\regimefactor}[0]{\lambda} 

\newcommand{\tb}[1][b]{t_{#1}} 
\newcommand{\edgecompleteset}[0]{M}
\newcommand{\Gnp}{\ensuremath{G(n,p)}} 
\newcommand{\Gknp}{\ensuremath{G^k(n,p)}} 
\newcommand{\Bin}[0]{\textrm{Bin}} 
\newcommand{\minp}[2]{\min_{#1} \left\{ #2\right\}} 
\newcommand{\maxp}[2]{\max_{#1} \left\{ #2\right\}} 

\newcommand{\estimerror}[0]{\delta}

\newtheorem{thm}{Theorem}[section]

\newtheorem{coro}[thm]{Corollary}
\newtheorem{lem}[thm]{Lemma}
\newtheorem{conjecture}[thm]{Conjecture}
\theoremstyle{remark}

\theoremstyle{definition}
\newtheorem{definition}[thm]{Definition}
\newtheorem{question}[thm]{Question}

\crefname{thm}{theorem}{theorems}
\crefname{prop}{proposition}{propositions}
\crefname{coro}{corollary}{corollaries}
\crefname{lem}{lemma}{lemmas}
\crefname{definition}{definition}{definitions}
\crefname{conjecture}{conjecture}{conjectures}
\crefname{question}{Question}{questions}

\newtheoremstyle{claim}
{}
{}
{\itshape}
{}
{\bf}
{.}
{.5em}
{}
\theoremstyle{claim}
\newtheorem{claim}{Claim}

\crefname{claim}{claim}{claims}

\title{Catching a robber on a random $k$-uniform hypergraph}
\author{Joshua Erde$^1$, Mihyun Kang$^1$}
\address{$^1$Institute of Discrete Mathematics, Graz University of Technology, Steyrergasse 30, 8010 Graz, Austria}
\email{\{erde,kang,schmid\}@math.tugraz.at}

\author{Florian Lehner$^2$}
\address{$^2$Department of Mathematics, University of Auckland, 38 Princes Street, 1010,  Auckland, New Zealand}
\email{florian.lehner@auckland.ac.nz}

\author{Bojan Mohar$^3$}
\address{$^3$Department of Mathematics, Simon Fraser University, 8888 University Drive, Burnaby, BC, Canada}
\email{mohar@sfu.ca}

\author{Dominik Schmid$^1$}


\numberwithin{equation}{section}
\begin{document}
\maketitle

\begin{abstract}
The game of \emph{Cops and Robber} is usually played on a graph, where a group of cops attempt to catch a robber moving along the edges of the graph. The \emph{cop number} of a graph is the minimum number of cops required to win the game. 
An important conjecture in this area, due to Meyniel, states that the cop number of an $n$-vertex connected graph is $O(\sqrt{n})$. In 2016, Pra\l at and Wormald [Meyniel's conjecture holds for random graphs, Random Structures Algorithms. 48 (2016), no. 2, 396–421. MR3449604] showed that this conjecture holds with high probability for random graphs above the connectedness threshold. Moreover, \L uczak and Pra\l at [Chasing robbers on random graphs: Zigzag theorem, Random Structures Algorithms. 37 (2010), no. 4, 516–524. MR2760362] showed that on a $\log$-scale the cop number demonstrates a surprising \emph{zigzag} behaviour in dense regimes of the binomial random graph $G(n,p)$. In this paper, we consider the game of Cops and Robber on a hypergraph, where the players move along hyperedges instead of edges. We show that with high probability the cop number of the $k$-uniform binomial random hypergraph $G^k(n,p)$ is $O\left(\sqrt{\frac{n}{k}}\, \log n \right)$ for a broad range of parameters $p$ and $k$ and that on a $\log$-scale our upper bound on the cop number arises as the minimum of \emph{two} complementary zigzag curves, as opposed to the case of $G(n,p)$. Furthermore, we conjecture that the cop number of a connected $k$-uniform hypergraph on $n$ vertices is $O\left(\sqrt{\frac{n}{k}}\,\right)$. 
\end{abstract}

\section{Introduction and results}\label{sec:Intro&Main}
\subsection{Motivation}\label{subsec:Motivation}

    The game of \emph{Cops and Robber} was introduced by Quilliot \cite{quilliot1978jeux} and independently by Nowakowski and Winkler \cite{NoWi1983}. It is a two-player game played on a simple connected graph $G= (V,E)$, with one player controlling a set of $m$ cops and the other player controlling a single robber. For convenience, we will sometimes refer to the cops and the robber as \emph{pieces}. At the start of the game, the first player chooses a starting vertex for each of the cops, then the second player chooses a starting vertex for the robber. Subsequently, the players take alternating turns and in each turn a player can move each of their pieces to an adjacent vertex (i.e., the pieces move along the edges of $G$). Note that more than one cop can simultaneously occupy a single vertex and that not every piece must be moved in every turn. The position of all the pieces is known to both players throughout the game. The cops win if at some point in the game a cop occupies the same vertex as the robber, otherwise the robber wins. As this is a game with full information, for each graph $G$ and each number of initial cops $m$, one of the two players has a winning strategy. The \emph{cop number $\cop{G}$} of a graph $G$ is defined as the minimum number $m \in \N$, such that $m$ cops have a winning strategy on $G$. The cop number has been extensively studied since the introduction of this game. 
    
    Whilst there is a structural characterisation of the graphs with cop number one \cite{NoWi1983}, in general the problem of determining the cop number of a graph is EXPTIME-complete \cite{Ki2015},
    and so research in this area has been focused on bounding the cop number of particular graph classes. For example, a classic result of Aigner and Fromme \cite{AiFr1984} shows that the cop number of a connected planar graph is at most three. More generally, it is known that the cop number is bounded for any proper minor-closed class of graphs \cite{A86}, and there has been much research into determining the largest cop number of a graph that can be embedded in a fixed surface \cite{BoErLePi2021,CFJT14,GHKMR21,IMW23,QUILLIOT1985,Schroeder2001}. 
    
    Perhaps the most well-known conjecture in this area is Meyniel's conjecture (communicated by Frankl \cite{FRANKL1987}).
    
    \begin{conjecture}\label{conj:Meyniel}
        Let\/ $G$ be a connected graph on $n$ vertices. Then $\cop{G} = O\left(\sqrt{n}\right)$.
    \end{conjecture}

    Despite much interest in this conjecture, there has been relatively little improvement to the trivial bound of $O(n)$. Frankl \cite{FRANKL1987} gave the first non-trivial upper bound on the cop number of $O\left(\frac{n \log \log n}{\log n}\right)$, and this bound was improved to $O\left(\frac{n}{\log n}\right)$ by Chiniforooshan \cite{Ch2008}. As of today, the best known general upper bound on the cop number is $n 2^{-(1+o(1)) \sqrt{\log n}}$, given independently by Lu and Peng \cite{LuPe2012} and by Scott and Sudakov \cite{ScSu2011}, which was later generalised to different variations of the Cops and Robber game by Frieze, Krivelevich and Loh \cite{FrKrLo2012}. We note that this bound is still $\Omega \left(n^{1-o(1)} \right)$, and it remains an open question as to whether the cop number can be bounded by $O\left(n^{1-\epsilon}\right)$ for any fixed $\epsilon>0$ \cite{BB12}.
    
     A natural step towards understanding Conjecture \ref{conj:Meyniel} is to consider the cop number of the \emph{binomial random} graph $\Gnp$. For $p$ constant, it was shown by Bonato, Hahn and Wang \cite{BHW07} that whp\footnote{Throughout the paper, all asymptotics are considered as $n \to \infty$ and so, in particular, whp (with high
probability) means with probability tending to one as $n \to \infty$.} the cop number of $\Gnp$ is logarithmic in $n$, and hence Conjecture \ref{conj:Meyniel} holds for almost all graphs.
        However, if we let $p$ vary as a function of $n$, then more interesting behaviour can be seen to develop. Indeed, \L uczak and Pra\l at \cite{Luczak_Pralat_RobberZigZag} showed that the cop number of $\Gnp$ behaves in a rather interesting manner in \emph{dense} regimes. Their result can be roughly summarised as follows, where we use $\Tilde{\Theta}(\cdot)$ to indicate  a bound which holds up to logarithmic factors.

    \begin{thm}[{\cite[Theorem 1.1]{Luczak_Pralat_RobberZigZag}}]\label{thm:LuPr:zigzag}
        Let\/ $0 < \alpha <1$ and $d = np = n^{\alpha +o(1)}$. 
        \begin{enumerate}
            \item If $\frac{1}{2j+1} < \alpha < \frac{1}{2j}$, for some $j \in \N$, then \whp
            \begin{align*}
                \cop{\Gnp} = \Theta\left(d^j\right).
            \end{align*}
            \item If $\frac{1}{2j} < \alpha < \frac{1}{2j-1}$, for some $j \in \N$, then \whp
            \begin{align*}
                \cop{\Gnp} = \Tilde{\Theta}\left(\frac{n}{d^j} \right).
            \end{align*}
        \end{enumerate}
    \end{thm}
    In particular, Theorem \ref{thm:LuPr:zigzag} implies that the function $f\colon (0,1) \to \mathbb{R}$, defined as

    \begin{align}
        f(x) = \frac{\log \left( \bar{c}\left(G \left(n,n^{x-1}\right) \right)\right)}{\log n}, \label{e:graphf}
    \end{align} where $\bar{c}$ denotes the median of the cop number, has a characteristic zigzag shape (see Figure \ref{fig:zigzag}).

    In particular, Theorem \ref{thm:LuPr:zigzag} implies that whp $\cop{\Gnp} = \tilde{O}\left(\sqrt{n}\right)$ throughout this range of $p$, and that conversely there are choices of $p$ where whp $\cop{\Gnp} = \tilde{\Theta}\left(\sqrt{n}\right)$ and Conjecture \ref{conj:Meyniel} is close to tight for almost all graphs of this density. Note that, although \Cref{thm:LuPr:zigzag} does not explicitly deal with the case $d = n^{\frac{1}{k} + o(1)}$ with $k \in \mathbb{N}$, the proofs for the upper bounds
    in \cite{Luczak_Pralat_RobberZigZag} also cover this case, and the authors indicate how to extend the lower bound to this range of $d$. Bollob\'{a}s, Kun and Leader  \cite{BoKuLe2013} gave a similar bound which holds also for sparser regimes of $p$.
    
    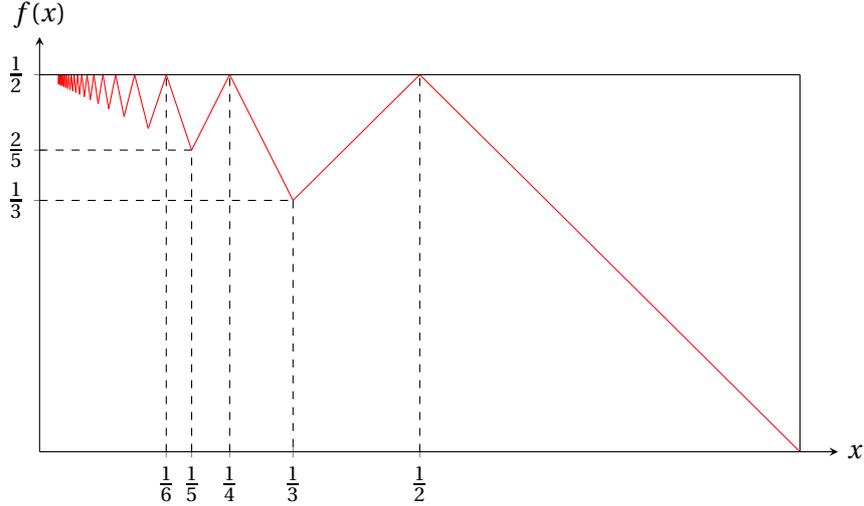
\begin{figure}[!ht]
        \centering
        \begin{tikzpicture}
    \begin{axis}[xmin=0, xmax=1.05, ymin=0, ymax = 0.55, y=10cm,
    x=10cm,
    axis lines = left,
    xtick = {0.1666,0.2,0.25,0.3333,0.5},
    xticklabels= {$\frac{1}{6}$,$\frac{1}{5}$,$\frac{1}{4}$,$\frac{1}{3}$,$\frac{1}{2}$},
    ytick = {0.3333,0.4,0.5},
    yticklabels= {$\frac{1}{3}$,$\frac{2}{5}$,$\frac{1}{2}$},
    ]
    \def\i{20} 
    \foreach \a in {1,...,\i}
       \addplot[red,domain=1/(2*\a):1/(2*\a-1)] {1-\a*x};
        \foreach \b in {1,...,\i}
       \addplot[red,domain=1/(2*\b +1):1/(2*\b)] {\b*x};

    \draw[dashed,opacity = 0.5] ({axis cs:0.5,0}|-{ axis cs:0,0}) -- ({axis cs:0.5,0}|-{axis cs:0,0.5});
    \draw[dashed,opacity = 0.5] ({axis cs:0.3333,0}|-{axis cs:0,0.3333})--({axis cs:0.3333,0}|-{ axis cs:0,0});
    \draw[dashed,opacity = 0.5] ({axis cs:0.25,0}|-{ axis cs:0,0}) -- ({axis cs:0.25,0}|-{axis cs:0,0.5});
    \draw[dashed,opacity = 0.5] ({axis cs:0.2,0}|-{ axis cs:0,0}) -- ({axis cs:0.2,0}|-{axis cs:0,0.4});
    \draw[dashed,opacity = 0.5] ({axis cs:0.1666,0}|-{ axis cs:0,0}) -- ({axis cs:0.1666,0}|-{axis cs:0,0.5});

     \draw[dashed,opacity = 0.5] ({axis cs:0,0}|-{ axis cs:0,0.3333}) -- ({axis cs:0.3333,0}|-{axis cs:0,0.3333});
     \draw[dashed,opacity = 0.5] ({axis cs:0,0}|-{ axis cs:0,0.4}) -- ({axis cs:0.2,0}|-{axis cs:0,0.4});

     \draw[-] ({axis cs:0,0}|-{ axis cs:0,0.5}) -- ({axis cs:1,0}|-{axis cs:0,0.5});
     \draw[-] ({axis cs:1,0}|-{ axis cs:1,0.5}) -- ({axis cs:1,0}|-{axis cs:1,0});

     \coordinate (A) at (axis cs:0,0.55);
     \coordinate (B) at (axis cs:1.05,0);
    \end{axis}
    \node[above] at (A) {$f(x)$};
    \node[right] at (B) {$x$};
    \end{tikzpicture}
        \caption{Zigzag shape of the function $f$}
        \label{fig:zigzag}
    \end{figure}  
    
    Meyniel's conjecture was finally resolved for all random graphs above the connectedness threshold by Pra\l at and Wormald \cite{Pralat_Wormald_CopsMeyniels}. In fact, their result holds for all random graphs with density above $\frac{1}{2} \log n$.
     
     \begin{thm}[\cite{Pralat_Wormald_CopsMeyniels}, Theorem 1.2]
         Let\/ $\epsilon >0$, and $p(n-1) \geq \left(\frac{1}{2} +\epsilon \right)\log n $. Then \whp
         \begin{align*}
             \cop{\Gnp} = O\left(\sqrt{n}\right). 
         \end{align*}
     \end{thm}
    
    In this paper we consider a variant of the Cops and Robber game on hypergraphs, and in particular $k$-uniform hypergraphs, which we call \emph{$k$-graphs}, for $k \in \N_{\geq 2}$.
    The game is defined analogously to the $2$-graph case, with the only difference  being that the pieces move along hyperedges instead of edges. For the sake of brevity, when it is clear from the context that we are talking about a hypergraph, we will refer to hyperedges as simply edges. Similarly to $2$-graphs, we define the \emph{cop number} of a hypergraph $H$ to be
    \[
        \cop{H} \coloneqq \minp{}{m \in \N \colon m \text{ cops have a winning strategy to catch a robber on } H}.
    \]
   This game was first considered by Gottlob, Leone and Scarcello \cite{GoLeSc2003} and by Adler \cite{Adler2004}. For more recent results on the hypergraph game we refer the reader to \cite{SiBoLiSi2021}, where some classic results on the cop number of $2$-graphs are generalised to this setting.

    Note that by replacing every edge in a hypergraph by a clique, we arrive at an equivalent $2$-graph game on the same vertex set. Thus, the game of Cops and Robber on hypergraphs is equivalent to the $2$-graph game played on a restricted class of graphs.
    On the other hand, we can transform a graph $G$ into a $2k$-uniform hypergraph $H$ with $\cop{G} = \cop{H}$ via a simple blow-up construction: We replace each vertex $v$ in $G$ by $k$ vertices $\{v_1,v_2,\ldots v_k\}$ and form a hypergraph $H = H(G)$ on $\{ v_i \colon v\in V(G), i \in [k]\}$ by taking an edge of the form $\{u_1,u_2,\ldots, u_k, v_1,v_2,\ldots,v_k\}$ for each edge $e=\{u,v\}$ of $G$ (see Figure \ref{fig:blowupconstruct}). It is then easy to check that $\cop{G} = \cop{H}$, and moreover $|V(H)| = k |V(G)|$.

    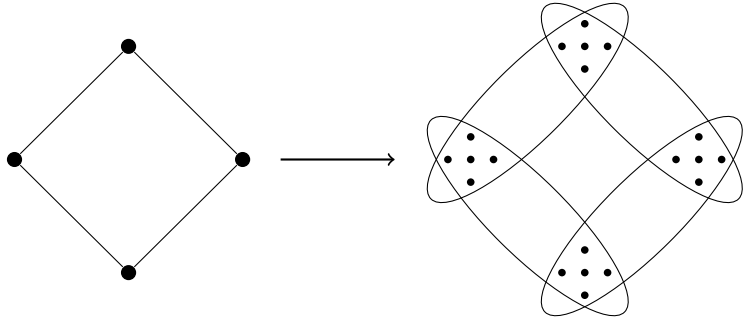
\begin{figure}
        \centering
        \begin{tikzpicture}
    \tikzstyle{small} = [circle, fill = black, minimum size=0.1cm, inner sep=0.2pt];
    \tikzstyle{occ} = [circle, fill = black, minimum size=0.2cm, inner sep=0.2pt];

    \coordinate (secondpic) at (6,0);
    \def\width{0.3};

    \node[occ] (v1) at (0,0) {};
    \node[occ] (v2) at (1.5,1.5) {};
    \node[occ] (v3) at (0,3) {};
    \node[occ] (v4) at (-1.5,1.5) {};
    
    \foreach \i/\j in {1/2,2/3,3/4,4/1}
    \draw[-] (v\i)--(v\j);

    \draw[->, thick] (2,1.5)--(3.5,1.5) ;

    \foreach \i in {1,...,4}
    \foreach \x/\y in {0/0,\width/0,0/\width,0/-\width,-\width/0}
    \node[small] at ($(v\i)+(secondpic)+(\x,\y)$) {};

    \foreach \i/\j/\x in {1/2/1,2/3/-1,3/4/1,4/1/-1}
    \draw ($0.5*(v\i)+0.5*(v\j)+(secondpic)$) ellipse [x radius = 1.8 cm, y radius = 0.5 cm, rotate = \x *45];
    
\end{tikzpicture} 
        \caption{An example of the blow-up construction to generate a $2k$-graph $H$ from a $2$-graph that has the same cop number. In this case, $k = 5$, $|V(H)| = 20$ and $\cop{H} = 2$. }
        \label{fig:blowupconstruct}
    \end{figure}
     From these two observations, it is easy to see that the following holds:
    \begin{align}
        \maxp{}{ \cop{G} \colon G \text{ a graph }, |V(G)|=\frac{2n}{k} } &\leq \maxp{}{ \cop{H} \colon H \text{ a }k\text{-graph }, |V(H)| = n } \label{e:graphtohypergraph}\\
        &\leq \maxp{}{ \cop{G} \colon G \text{ a graph }, |V(G)|=n } \nonumber.
    \end{align}
    
    In particular, as there are graphs with $\cop{G} = \Omega\left(\sqrt{n}\right)$, there are also $k$-graphs with $\cop{H} = \Omega\left(\sqrt{\frac{n}{k}}\right)$. It would seem surprising that such a simple construction, which is essentially graphical in nature, could capture the worst case behaviour for the cop number in hypergraphs of higher uniformity, but we conjecture that this bound is in fact tight. 
    \begin{conjecture}\label{conj:MeynielHyper}
        Let $H$ be a connected $k$-graph on $n$ vertices. Then $\cop{H} = O\left(\sqrt{\frac{n}{k}}\,\right)$.
    \end{conjecture}

Clearly Conjecture \ref{conj:MeynielHyper} is a generalisation of Meyniel's conjecture, but we note further that for any polynomial function $k:=k(n) = n^{\alpha}$ with $\alpha<1$, Conjecture \ref{conj:MeynielHyper} implies Meyniel's conjecture. Indeed, if $\cop{H} = O\left(\sqrt{\frac{n}{k(n)}}\right) = O\left(n^{\frac{1-\alpha}{2}} \right)$  for all $n$-vertex $n^{\alpha}$-graphs, then by \eqref{e:graphtohypergraph}
  \[
\maxp{}{ \cop{G} \colon G \text{ a graph }, |V(G)|= 2m } \leq \maxp{}{ \cop{H} \colon H \text{ a }m^{\frac{\alpha}{1-\alpha}}\text{-graph }, |V(H)| = m^{\frac{1}{1-\alpha}} } \leq  O\left(\sqrt{m} \right). 
\]

    As with Meyniel's Conjecture, a first step towards Conjecture \ref{conj:MeynielHyper} is to consider the behaviour of the cop number of \emph{random} $k$-graphs.

\subsection{Main results}\label{subsec:Main}

The \emph{$k$-uniform binomial random hypergraph}, which we denote by $\Gknp$, is a random $k$-graph with vertex set $[n]$ in which each edge, that is, each subset of $[n]$ of size $k$, appears independently with probability $p$. Although the main focus of this paper is $\Gknp$, the strategies we develop for the cops work in a more general class of $k$-graphs, those satisfying certain expansion properties. 

Very roughly, if we denote by $N^r_V(v)$ the vertices that are at most at a fixed distance $r$ from $v$, then in $\Gknp$ we expect this set to be growing exponentially quickly in $r$, with its size tightly concentrated around its expectation. Furthermore, for different vertices $v$ and $w$ we do not expect the neighbourhoods $N^r_V(v)$ and $N^r_V(w)$ to have a large intersection, and so, for small subsets $A \subseteq [n]$ we expect the number of vertices at most at a fixed distance $r$ from $A$ to be around $|A|$ times the size of $N^r_V(v)$. Similarly, we expect the set of edges $N^r_E(v)$ at most at a fixed distance $r-1$ from $v$ to be growing at some uniform exponential rate, and for ranges of $p$ where the random hypergraph is sparse enough, and so few pairs of edges have a large intersection, this rate of growth should be roughly $\frac{1}{k}$ times that of the vertex-neighbourhoods. 

Informally, given $\smoothconst>0$ we say that a graph is \smoothexp if the sizes of its vertex and edge-neighbourhoods have this uniform exponential growth, up to some multiplicative error in terms of $\smoothconst$. See Definition \ref{Property:neighbourhood_estimates} for a precise definition of this notion.

Our first result supports Conjecture \ref{conj:MeynielHyper} up to a $\log$-factor for $k$-graphs that are \smoothexp for a fixed expansion constant $\smoothconst$. 
\begin{thm}\label{thm:main:Meyniel}
Let $k \in \N_{\geq 2}$, let $\smoothconst >0$ and let $G$ be a \smoothexp $k$-graph on $n$ vertices. Then
\[
\cop{G} \leq 20 \smoothconst^{-2} \sqrt{\frac{n}{k}}\, \log n .
\]
\end{thm}

In fact, depending on the relationship between the average degree $d$ of the hypergraph, its uniformity $k$ and its order $n$, we can give a more refined bound for the cop number of a \smoothexp hypergraph.

\begin{thm}\label{thm:main:regimes}
Let $k \in \N_{\geq 2}, \xi \in (0,1]$ be fixed and let $G$ be a \smoothexp $k$-graph on $n$ vertices with average vertex degree\footnote{Here the degree of a vertex $v$ is the number of vertices which share an edge with $v$, rather than the number of edges containing $v$.} $d=d(G)$.
For all $j \in \N$ the cop number of $G$ satisfies the following.

\begin{enumerate}   
    \item\label{r:a} If\/ $n^\frac{1}{2j+1} \leq d \leq \left(\frac{n}{k}\right)^\frac{1}{2j}$, then with $\regimefactor = \left\lceil \frac{n}{d^{2j+1}} \log n  \right\rceil$,
    \begin{equation*}\label{thm:main:regimes:n(2j+1)}
    \cop{G}  \leq 20 \smoothconst^{-2} d^{j} \regimefactor.
    \end{equation*}
    
    \item\label{r:b} If\/ $\left(\frac{n}{k}\right)^\frac{1}{2j} \leq d \leq n^\frac{1}{2j}$, then 
    \begin{equation*}\label{thm:main:regimes:(n/k)(2j)}
    \cop{G} \leq 20 \smoothconst^{-1} \frac{n}{k d^j } \log n .
    \end{equation*}

     \item\label{r:c} If\/ $n^\frac{1}{2j} \leq d \leq (nk)^\frac{1}{2j}$, then with $\regimefactor = \maxp{}{\left\lceil \frac{n}{d^{2j}} \log n  \right\rceil ,\left\lceil \frac{k}{d^{j}} \log n  \right\rceil}$,
    \begin{equation*}\label{thm:main:regimes:n(2j)}
    \cop{G} \leq  20 \smoothconst^{-2} \frac{d^{j}}{k} \regimefactor.
    \end{equation*}

    \item\label{r:d} If\/ $(nk)^\frac{1}{2j} \leq d \leq n^\frac{1}{2j-1}$, then 
    \begin{equation*}\label{thm:main:regimes:(nk)(2j)}
    \cop{G} \leq 20 \smoothconst^{-1} \frac{n}{d^{j}} \log n .
    \end{equation*}
\end{enumerate}
\end{thm}

Let us explain in more detail the bounds in Theorem \ref{thm:main:regimes}. In general, the upper bounds in \Cref{thm:main:regimes} are increasing in $d$ in the regimes \eqref{r:a} and \eqref{r:c} and decreasing in $d$ in the regimes \eqref{r:b} and \eqref{r:d}. In particular, and perhaps surprisingly, if we fix $k$ and $n$ and vary $d$, in certain regimes increasing the average degree, and hence the number of edges, can help the cops, and in other regimes increasing the number of edges can help the robber.

We note that some of the regimes of Theorem \ref{thm:main:regimes} can `collapse' if the left border of a regime is larger than its right border. Specifically, this can happen in regime \eqref{r:a}, if we have $d \leq n^\frac{1}{2j}$ and 
$n^\frac{1}{2j+1} > \left(\frac{n}{k}\right)^\frac{1}{2j}$ or equivalently $k > n^\frac{1}{2j+1}$ and it can happen in regime \eqref{r:d}, if $d \leq n^\frac{1}{2j-1}$ and $(nk)^\frac{1}{2j} > n^\frac{1}{2j-1}$, or equivalently $k > n^\frac{1}{2j-1}$.
However, under the reasonable assumption that $G$ is connected, we have $k \leq d$, and so this second case does not occur, and the first case only occurs if $n^\frac{1}{2j+1} \leq k \leq d \leq n^\frac{1}{2j}$ holds for some $j \in \N$. In this case, regime \eqref{r:a} collapses and the cop number is bounded as in \eqref{r:b}.
Furthermore, we note that the second argument of the maximum in the definition of $\regimefactor$ in regime \eqref{r:c} is only relevant in the special case where $k \geq \sqrt{n}$. In that case, the factor takes its largest value of $\log n$ for the smallest possible value of $d$, namely $d=k$. For increasing $d$, the factor then decreases until $d$ is larger than $k$ by a $\log$-factor, at which it attains its smallest value of $1$. Otherwise, in regimes \eqref{r:a} and \eqref{r:c}, the factor $\regimefactor$ takes its largest value of $\log n$ at the left border of the respective regimes of $d$. Again, it then decreases with increasing $d$ until it takes the value of $1$, which happens as soon as $d$ is bound away from the left border by a sufficiently large multiplicative factor ($\log^{\frac{1}{2j+1}} n$ in \eqref{r:a} and $\log^{\frac{1}{2j}} n$ in \eqref{r:c}).

Our final result shows that whp $\Gknp$ satisfies the desired expansion properties as long as $k$ is growing with $n$ and $p$ is not too small.

\begin{thm}\label{thm:main:Gknp}
There exists a universal constant $\smoothconst > 0$ such that if $k = k(n), p=  p(n)>0$ are such that $k = \omega(\log n)$ and $\frac{n}{k} \geq p \binom{n-1}{k-1} = \omega\left(\log^3 n\right)$, then whp $\Gknp$ is \smoothexp. 
\end{thm}
Let us give some intuition for the conditions on $k$ and $p$. The value $p \binom{n-1}{k-1}$ is roughly the expected number of edges every vertex in $\Gknp$ meets, and the lower bound on this quantity ensures that we can assume this is concentrated around its expectation. On the other hand, $p k \binom{n-1}{k-1}$ is roughly the expected degree of a vertex in $\Gknp$, and so it is natural to restrict this to be at most $n$. 

Note that it follows from Theorems \ref{thm:main:Meyniel} and \ref{thm:main:Gknp} that Conjecture \ref{c:main} holds for the same range of $n,p$ and $k$ up to polylogarithmic factors. 

\begin{coro}\label{c:main}
If\/ $k = k(n), p=  p(n)>0$ are such that $k = \omega(\log n)$ and $\frac{n}{k} \geq p \binom{n-1}{k-1} = \omega(\log^3 n)$, then whp $c\left(\Gknp\right) = \tilde{O}\left( \sqrt{\frac{n}{k}}\,\right)$.
\end{coro}

Furthermore, \Cref{thm:main:Gknp} allows us to apply Theorem \ref{thm:main:regimes} to $\Gknp$ for a broad range of parameters and we can bound $c\left(\Gknp\right)$ more precisely in certain ranges. It turns out that a sensible parameterisation to take is as follows. Let us define $\hat{d} = \hat{d}(n,p,k) \coloneqq pk \binom{n-1}{k-1}$, which is roughly the expected degree of a vertex in $\Gknp$ and let $\hat{d} = n^\alpha$ and $ k = n^\beta$ for some $0 < \beta \leq \alpha \leq 1$.  We consider the function $f_\beta \colon (\beta,1) \to \R$ defined as
\[
f_\beta(\alpha) \coloneqq \frac{\log \left(\bar{c} \left(\Gknp \right)\right)}{\log n},
\] 
with $\bar{c}$ being the upper bound for the cop number obtained from Theorem \ref{thm:main:regimes}. It follows that $f_\beta$ again has a characteristic zigzag shape, see Figure \ref{fig:hyperzigzag}. In contrast to the case of $\Gnp$ (see Figure \ref{fig:zigzag}), the zigzag shape in $\Gknp$ arises as the intersection of two complementary zigzags, coming from two different strategies, and so has twice as many peaks and troughs. In particular, it can be seen that $f_\beta(\alpha) \leq (1+o(1))\frac{1-\beta}{2}$ for all $\alpha \in (\beta,1)$, corresponding to the bound of $\Tilde{O}\left(\sqrt{\frac{n}{k}}\,\right)$ on the cop number.

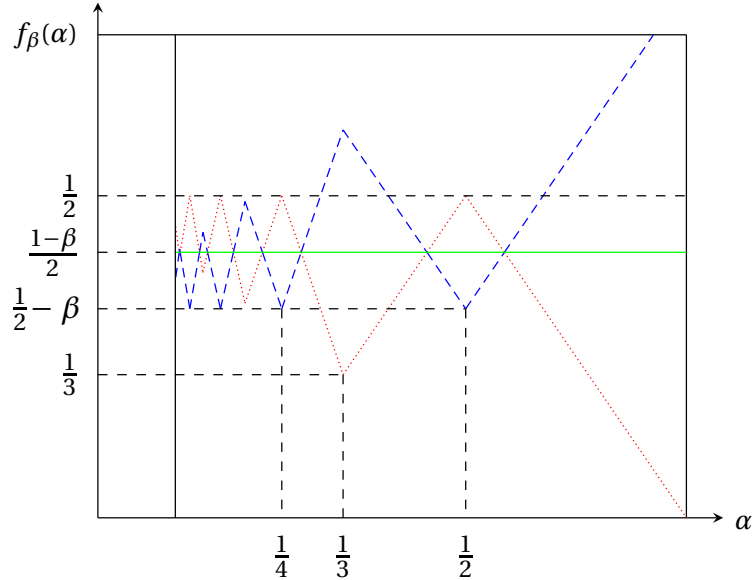
\begin{figure}[ht!]
\begin{tikzpicture}[scale=1.2]
    \def\vbeta{2/19} 
    \def\i{5} 
    \def\opp{0.4} 
    
    \begin{axis}[xmin=0, xmax=0.85, ymin= 0.2, ymax = 0.68,
    axis lines = left,
	xtick = {0.25,0.3333,0.5},
    xticklabels= {$\frac{1}{4}$,$\frac{1}{3}$,$\frac{1}{2}$},
    xtick style={draw=none},
    ytick = {0.333334,0.5-\vbeta, 0.5-0.5*\vbeta,0.5},
    yticklabels= {$\frac{1}{3}$,$\frac{1}{2}-$ \small{$\beta$},$\frac{1-\beta}{2}$,$\frac{1}{2}$},
    ytick style={draw=none},
    ]

	\draw[-] ({axis cs:0,0}|-{ axis cs:0,0.65}) -- ({axis cs:0.8,0}|-{axis cs:0,0.65});
    \draw[-] ({axis cs:0.8,0}|-{ axis cs:0.8,0.65}) -- ({axis cs:0.8,0}|-{axis cs:0.8,0});

	  \draw[dashed,opacity = 0.5] ({axis cs:0.25,0}|-{ axis cs:0,0.5-\vbeta}) -- ({axis cs:0.25,0}|-{axis cs:0,0}); 
	 
	  \draw[dashed,opacity = 0.5] ({axis cs:0.3333,0}|-{axis cs:0,0.3333})--({axis cs:0.3333,0}|-{ axis cs:0,0}); 
    
     \draw[dashed,opacity = 0.5] ({axis cs:0.5,0}|-{axis cs:0,0.5-\vbeta})--({axis cs:0.5,0}|-{ axis cs:0,0});
     
    \foreach \a in {1,...,\i}
    {
     \addplot[densely dotted, opacity = \opp, red,domain=1/(2*\a):(1+\vbeta)/(2*\a)] {1-\a*x};       
       \addplot[densely dotted, red,domain=(1+\vbeta)/(2*\a):1/(2*\a-1)] {1-\a*x};
  
       }
        \foreach \b in {1,...,\i}
        {
       \addplot[densely dotted, red,domain=1/(2*\b +1):(1-\vbeta)/(2*\b)] {\b*x};
       \addplot[densely dotted, opacity = \opp, red,domain=(1-\vbeta)/(2*\b):1/(2*\b)] {\b*x};
       }
       
       \foreach \a in {2,...,\i}
       {
       \addplot[densely dashed, blue, domain=1/(2*\a):(1+\vbeta)/(2*\a)] {\a*x - \vbeta};
       \addplot[densely dashed, blue, opacity = \opp, domain=(1+\vbeta)/(2*\a):1/(2*\a-1)] {\a*x - \vbeta};
       }
       \addplot[densely dashed, blue, domain=1/(2):(1+\vbeta)/(2)] {x - \vbeta};
       \addplot[densely dashed, blue, opacity = \opp, domain=(1+\vbeta)/(2):0.755] {x - \vbeta};
        \foreach \b in {1,...,\i}
        {
       \addplot[densely dashed, opacity = \opp, blue,domain=1/(2*\b +1):(1-\vbeta)/(2*\b)] {1 - \b*x - \vbeta};
       \addplot[densely dashed,  blue,domain=(1-\vbeta)/(2*\b):1/(2*\b)] {1 - \b*x - \vbeta};
       }
       
     \addplot[green,domain=\vbeta:0.8] {1/2-\vbeta/2}; 
     
     \fill [white] (0.04,0.31) rectangle (\vbeta,0.6);
     \draw [-] (\vbeta,0)--(\vbeta,0.65); 
     
	\draw[dashed,opacity = 0.5] ({axis cs:0,0}|-{ axis cs:0,0.3333334}) -- ({axis cs:0.3333334,0}|-{axis cs:0,0.3333334});
	
	\draw[dashed,opacity = 0.5] ({axis cs:0,0}|-{ axis cs:0,0.5-\vbeta}) -- ({axis cs:0.5,0}|-{axis cs:0,0.5-\vbeta});
	
	\draw[dashed,opacity = 0.5] ({axis cs:0,0}|-{ axis cs:0,0.5-0.5*\vbeta}) -- ({axis cs:\vbeta,0}|-{axis cs:0,0.5-0.5*\vbeta});
	
	\draw[dashed,opacity = 0.5] ({axis cs:0,0}|-{ axis cs:0,0.5}) -- ({axis cs:0.8,0}|-{axis cs:0,0.5});

     \coordinate (A) at (axis cs:0.01,0.75);
     \coordinate (B) at (axis cs:1.05,0.208);
    \end{axis}
    \node[left] at (A) {$f_\beta(\alpha)$};
    \node[right] at (B) {$\alpha$};
\end{tikzpicture}
\caption{Alternating zigzag shape of the function $f_\beta(\alpha)$ for $\beta = \frac{2}{19}$. The blue (dashed) line is the upper bound coming from the edge strategy, the red (dotted) line is the upper bound coming from the vertex strategy. As can be seen, the two strategies give rise to two alternating zigzag shapes, that together make up the single zigzag with increased frequency. We note the worst bounds occur at the intersection points of the two lines, which all lie on the green (solid) line at $\frac{1-\beta}{2}$.}
\label{fig:hyperzigzag}
\end{figure}

\subsection{Techniques}\label{sec:Methods}

To give a lower bound for the cop number we need to exhibit a strategy for the cops. As in the work of {\L}uczak and Pra{\l}at\cite{Luczak_Pralat_RobberZigZag}, we show the existence of a strategy for the cops to \emph{surround} the robber using a probabilistic argument. Whilst in \cite{Luczak_Pralat_RobberZigZag} the strategies focused solely on surrounding a small \emph{vertex}-neighbourhood of the robber, we also consider a second type of strategy which aims to surround a small \emph{edge}-neighbourhood, and utilise both these strategies in our result. 

Assuming the robber starts on a vertex $v$, after his first $r$ moves the robber has to be in the $r$-th vertex-neighbourhood $\vNhd[r]{v}$, and specifically in some edge of the $r$-th edge-neighbourhood $\eNhd[r]{v}$. The cops aim to occupy each edge in $\eNhd[r]{v}$ before the robber has had time to leave this set. Since the cops move first and a cop can catch the robber in a single move once they occupy the same edge, the cops need to occupy each edge in $\eNhd[r]{v}$ within their first $r$ moves (see \Cref{figure:heuristic:edges}). The strategy of surrounding via vertices works similarly, the only difference being that the cops surround the $r$-th vertex-neighbourhood and have $r+1$ moves before the robber can escape. The pay-off in choosing to surround via vertices or edges can be seen as follows -- in the former we can use cops at a larger distance, and so in general we will have more cops to work with, whereas in the latter, since each edge contains many vertices, we will not have to occupy as many edges as we would have vertices, and so perhaps we can catch the robber with fewer cops.

For a fixed vertex $v$ and a fixed distance $r$, the existence of such a strategy can then be reduced to a \emph{matching problem} -- for instance in the case of the edge strategy, for each edge $e$ at distance at most $r$ from $v$ we need to assign a unique cop at distance at most $r$ from $e$, whose strategy is to occupy $e$ within the first $r$ turns of the game. We aim to show (see Claims \ref{claim:injection_vertexstrat}, \ref{claim:injection_edgestrat}) that such an assignment of cops can be found with \emph{positive} probability if we choose a \emph{random} set of cops, assigning a cop to each vertex in the graph independently with some probability $q$ (see  \Cref{figure:heuristic:edges}).

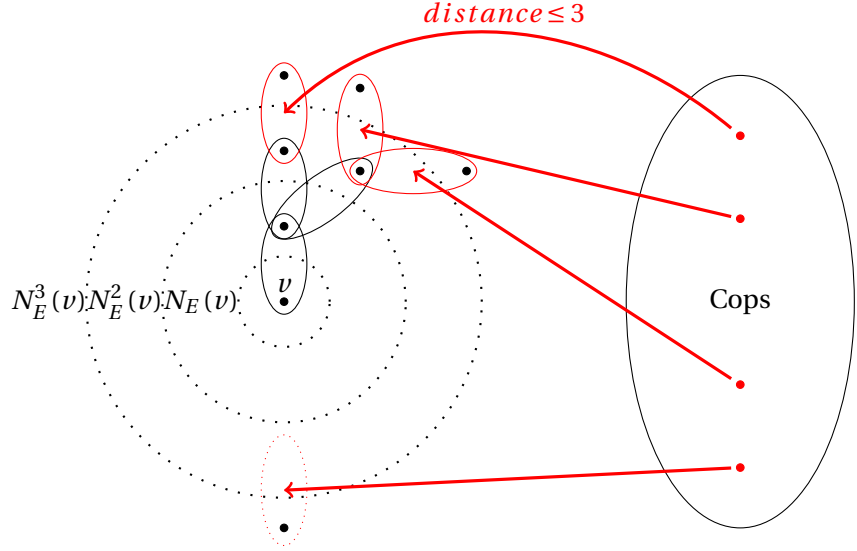
\begin{figure}[ht!]
\begin{tikzpicture}
	\def\circlerad{1};
    \def\vertexsize{0.05};
    \def\ellipse_x{(0.3)* \circlerad};
    \def\ellpise_y{(2/3)*\circlerad};
    \coordinate (startpoint) at (0,0);
    \coordinate (cops) at (6,0);
    
    \draw[fill=black]  (startpoint) circle [radius = \vertexsize] node [above]{$v$} ; 
    \draw[thick, loosely dotted] (startpoint) circle[radius = 0.6*\circlerad] node at ($(startpoint)-(1.1*\circlerad,0)$) {$\eNhd[1]{v}$}; 
    \draw[thick, loosely dotted] (startpoint) circle[radius = 1.6*\circlerad]node at ($(startpoint)-2.1*(\circlerad,0)$) {$\eNhd[2]{v}$};
    \draw[thick, loosely dotted] (startpoint) circle[radius = 2.6*\circlerad]node at ($(startpoint)-3.1*(\circlerad,0)$) {$\eNhd[3]{v}$}; 
        
    \draw ($(startpoint)+(0,\circlerad/2) $)ellipse [x radius = \ellipse_x  cm, y radius = \ellpise_y cm];
    \draw[fill=black]  ($(startpoint)+(0,\circlerad) $) circle [radius = \vertexsize] node (n1) {};
    \draw ($(startpoint)+(0, 1.5* \circlerad) $)ellipse [x radius = \ellipse_x  cm, y radius = \ellpise_y cm];
    \draw[fill=black]  ($(startpoint)+(0,2*\circlerad) $) circle [radius = \vertexsize]node (n2) {};
    \draw [red]($(startpoint)+(0, 2.5* \circlerad) $)ellipse [x radius = \ellipse_x  cm, y radius = \ellpise_y cm];
    \draw[fill=black]  ($(startpoint)+(0,3*\circlerad) $) circle [radius = \vertexsize]node (n3) {};
	\draw[fill=black]  ($(startpoint)-(0,3*\circlerad) $) circle [radius = \vertexsize]node (n4) {};

    \draw (cops) ellipse [x radius = 1.5, y radius = 3] node {Cops};
    \draw [red,fill = red]($(cops)+(0,2.2)$) circle [ radius = \vertexsize] node (c1) {};
    \draw [red,fill = red]($(cops)+(0,1.1)$) circle [ radius = \vertexsize] node (c2) {};
    \draw [red,fill = red]($(cops)-(0,1.1)$) circle [ radius = \vertexsize] node (c3) {};
    \draw [red,fill = red]($(cops)-(0,2.2)$) circle [ radius = \vertexsize] node (c4) {};
    
    
    \draw[fill=black]  ($(startpoint)+(60:2*\circlerad) $) circle [radius = \vertexsize] node (anchor) {};
	\draw ($0.5*(n1)+0.5*(anchor)$) ellipse [x radius = \ellipse_x  cm, y radius = 1.2 *\ellpise_y cm, rotate = 127];
	\draw[fill=black]  ($(anchor)+(0,1.1*\circlerad) $) circle [radius = \vertexsize] node (anchor2) {};
	\draw [red] ($0.5*(anchor)+0.5*(anchor2)$) ellipse [x radius = \ellipse_x  cm, y radius = 1.1 *\ellpise_y cm];
	\draw[fill=black]  ($(anchor)+(1.4*\circlerad,0)$) circle [radius = \vertexsize] node 	(anchor3) {};
	\draw [red]($0.5*(anchor)+0.5*(anchor3)$) ellipse [x radius = \ellipse_x  cm, y radius = 1.25 *\ellpise_y cm, rotate = 90];

\draw [red,dotted] ($(startpoint)-(0,2.5*\circlerad) $) ellipse [x radius = \ellipse_x  cm, y radius = 1.1 *\ellpise_y cm];

\draw[red, ->,very thick] (c1) to [out=140,in=45] node[midway,above] {$distance \leq 3$}($(startpoint)+(0, 2.5* \circlerad) $) ;
\draw[red, ->,very thick] (c2) to  ($0.5*(anchor)+0.5*(anchor2)$);
\draw[red, ->,very thick] (c3) to  ($0.5*(anchor)+0.5*(anchor3)$);
\draw[red, ->,very thick] (c4) to  ($(startpoint)-(0,2.5*\circlerad) $);
\end{tikzpicture}
\caption{A visualisation of the edge-surrounding strategy. The cops try to cover all edges of the third edge-neighbourhood of $v$ in $3$ moves, which is possible if there is a matching between $\eNhd[3]{v}$ and all vertices occupied by cops within distance three of these edges, which covers all edges of $\eNhd[3]{v}$.}
\label{figure:heuristic:edges}
\end{figure}

Assuming that our $k$-graph $G$ is \smoothexp, we have quite good control over the sizes of $\vNhd[r]{v}$ and $\eNhd[r]{v}$, and also over the number of vertices at a fixed distance from each vertex and edge contained in these sets. Using some standard probabilistic and combinatorial tools, we can show that for an appropriate choice of $q$, with positive probability we can find an appropriate assignment of cops for \emph{each} possible starting vertex $v$, and bound the number of cops $m$ we use in such a strategy, which in general depends not only on $r$, but also on the uniformity $k$ and average degree $d$ of $G$.

This leads to a family of bounds on the cop number, one for each $r\in \mathbb{N}$, for both the vertex and edge surrounding strategy. For a fixed choice of parameters $k$ and $d$, we then have to solve an integer optimisation problem to find which choice of $r$ (and of a vertex or edge surrounding strategy) leads to the best bound on the cop number, from which we can derive the bounds in Theorem \ref{thm:main:regimes}.

\subsection{Outline of the paper}\label{subsec:Outline}
The rest of the paper is structured as follows. In Section \ref{sec:Preliminaries} we introduce some notation and important definitions and state some auxiliary results. 
In Section \ref{sec:Proofs:regime&meyniel} we prove Theorems \ref{thm:main:Meyniel} and \ref{thm:main:regimes} and in Section \ref{sec:proof:Gknp} we prove Theorem \ref{thm:main:Gknp}. We conclude in Section \ref{sec:Discussion} by discussing some unresolved questions and possible directions for future research.

\section{Preliminaries}\label{sec:Preliminaries}

All asymptotics in the paper are taken as $n$ tends to infinity. We say that a statement $A(n)$ holds \textit{with high probability} (\textit{\whp} for short) if $\lim_{n \to \infty} \prob{A(n)} = 1$. We use standard Landau notation for all asymptotics. Furthermore, we omit floors or ceilings in proofs to improve readability. Let $\N_0 = \N \cup \{0\}$ and $\N_{\geq 2} = \{2,3,\ldots\}$. Throughout the paper we assume $k \in \N_{\geq 2}$.

Let $G$ be a $k$-uniform hypergraph, or $k$-graph, on vertex set $[n]$. The \emph{distance} between two different vertices $\dist{x}{y}$ is equal to the smallest length of a loose path between them, that is, the smallest $t \in \N$ such that there exists a sequence $e_1,e_2,\ldots, e_t$ of edges such that $x \in e_1, y \in e_t$ and $e_i \cap e_{i+1} \neq \emptyset$ for all $1\leq i \leq t-1$. Furthermore, we define $\dist{x}{x} = 0$, for all $x \in V(G)$. The distance between two sets of vertices $A$ and $A'$ is defined as the minimum distance between all pairs of vertices from the two sets:
\[
\dist{A}{A'} \coloneqq \min_{u \in A} \min_{v \in A'} \dist{u}{v}.
\]
For a set of vertices $A \subseteq V(G)$ and $r \in \N_0$ we denote the \emph{closed $r$-th vertex-neighbourhood of $A$} by $\vNhd{A}$.
Formally, 
\[
\vNhd[r]{A} \coloneqq \{v\in V(G) : \dist{A}{v} \leq r\}.
\]
For a set of edges $B \subseteq E(G)$ we write 
\[
\vine{B} \coloneqq \{v \in e : e\in B\}
\]
for the set of all vertices contained in at least one of the edges in $B$. 
We extend the definition of the closed $r$-th vertex-neighbourhood to sets of edges in the obvious way by setting
\[
\vNhd[r]{B} \coloneqq \vNhd[r]{\vine{B}} = \bigcup_{v\in\vine{B}} \vNhd[r]{v}.
\]
Furthermore, for a set of vertices $A \subseteq V(G)$ and $r \in \N$, we define the \emph{closed $r$-th edge-neighbourhood of $A$} as

\[
\eNhd[r]{A} \coloneqq \left\{e \in E(G) : \dist{e}{A} \leq r-1    \right\}.
\]
Note that to be included in the $r$-th edge-neighbourhood, an edge has to be within distance $r-1$ of the respective  vertex set. The motivation behind this parameter shift is that the amount of \lq discovered\rq \ vertices is the same for the $r$-th vertex- and edge-neighbourhood, or in other words $\vNhd[r]{A} = \vine{\eNhd[r]{A}}$.
For notational convenience, we will omit the superscript for the \emph{first} vertex- and edge-neighbourhood and write $\vNhd[1]{A} \coloneqq N_V^1(A)$ and $\eNhd[1]{A} \coloneqq N_E^1(A)$, respectively.

We denote the \emph{average vertex degree} of $G$ by
\[ 
d(G) \coloneqq \frac{1}{|V(G)|} \sum_{v \in V} |\vNhd[1]{v}|,
\]
where we consider the \emph{degree} of a vertex $v$ to be $|N_V(v)|$, and not $|N_E(v)|$ which is also sometimes referred to as the degree of a vertex in a hypergraph.

Note that the average vertex degree of $\Gknp$, i.e., $d \coloneqq d\left(\Gknp\right)$, is a 
random variable that is 
concentrated around its expectation, which is given as 
\begin{equation}\label{eq:average_degree_def}
\expec{d} = (n-1)\left(1-(1-p)^{\binom{n-2}{k-2}}\right).  
\end{equation}
For appropriate ranges of $p$ and $k$ we can estimate $\expec{d}$ as follows:
\begin{equation}\label{eq:average_degree_refinement}
\expec{d} =  (n-1)\left(1-(1-p)^{\binom{n-2}{k-2}}\right) =  n p\binom{n-2}{k-2} \left(1+o(1) \right) =  p k  \binom{n-1}{k-1} \left(1+o(1) \right),
\end{equation}
where the penultimate equation holds when $p\binom{n-2}{k-2} = o(1)$ and the final equation holds for $k = o(n)$. In the context of Cops and Robber games on random hypergraphs, these are reasonable assumptions, as otherwise either each edge-neighbourhood (deterministically) or each vertex-neighbourhood (in expectation) covers a constant fraction of the vertex set. If the graph is also connected, it can be shown in both cases that the cop number is then whp at most logarithmic in $n$.
We will also restrict ourselves to the case $d\geq k \geq 2$, since for $d <k$ the hypergraph will contain isolated vertices. Since the cop number is additive over disjoint unions, it is natural to restrict our attention to connected hypergraphs.

For convenience, rather than working with $d$, we will work with the following explicit quantity 
\begin{equation}\label{eq:average_degree_hat}
\dhat = \dhat(n,p,k) \coloneqq p k  \binom{n-1}{k-1}.
\end{equation}
We will show later (see \Cref{l:vertex_expansion}) that the size of the first vertex-neighbourhood of every vertex in $\Gknp$ lies close to $\dhat$, from which it then follows, that $d$ is approximately $\dhat$. 

\begin{definition}\label{Property:neighbourhood_estimates}
Let $G$ be a $k$-graph on $n$ vertices with average vertex degree $d\geq k$. Given a positive constant $0< \smoothconst \leq 1$, which we call the \emph{expanding constant}, we say $G$ is \smoothexp if $G$ has the following properties.

\begin{enumerate}[label = \bf{(A.\arabic*)}]
    \item\label{Property:neighbourhood_estimates:eNhd} For every vertex $v \in V(G)$ and $r \in \N$ satisfying $d^r \leq \sqrt{nk}$,
    \begin{equation*}
   \left|\eNhd[r]{v}\right| \leq \frac{1}{\smoothconst} \frac{d^r}{k}.
    \end{equation*}
    
    \item\label{Property:neighbourhood_estimates:vNhd} For every subset $A \subseteq V(G)$ of vertices and $r \in \N$,
    \begin{equation*}
    \smoothconst \minp{}{|A| d^r, n}\leq \left|\vNhd[r]{A}\right| \leq \frac{1}{\smoothconst} |A| d^r.
    \end{equation*}
      
    \item\label{Property:neighbourhood_estimates:vNhd:edges} For every subset $B \subseteq E(G)$ of edges and $r \in \N$,
    \begin{equation*}
    \smoothconst \minp{}{|B| kd^r, n}\leq \left|\vNhd[r]{B}\right|.    
    \end{equation*}

\end{enumerate}    
\end{definition}

Throughout the paper we use the following corollaries of the Chernoff bounds (see for example  \cite[Theorem 2.1, Corollary 2.3]{Janson_Luczak_Rucinski_randomgraphs}): 

\begin{thm}\label{thm:Chernoff}
Let $X \sim \Bin(n,p)$. Then for any $t>0$, we have

\begin{equation}\label{thm:Chernoff:|x-Ex|}
\prob{\,|X - \expec{X}| \geq t\,} \leq 2 \exp \left( -\frac{t^2}{2\left(\expec{X}+t/3\right)}\right),
\end{equation}
and 
\begin{equation}\label{thm:Chernoff:onesided}
\prob{\,X \leq \expec{X} - t\,} \leq \exp \left(-\frac{t^2}{2 \expec{X}} \right).
\end{equation}
In particular, if\/ $a \leq 10 \expec{X}$, then
\begin{equation}\label{thm:Chernoff:10Ex}
\prob{X \leq a} \leq \exp \left(-4a \right).
\end{equation}
\end{thm}

An important step in the proof of the main theorems is constructing matchings in specific bipartite graphs which cover one partition class. To this end we use Hall's marriage theorem, which we state here for the sake of completeness. 
\begin{thm}[\cite{Hall}, Theorem 1]\label{thm:Hall}
Let $(A\cup B,E)$ be a bipartite graph. The following two statements are equivalent.

\begin{enumerate}
    \item There is a matching covering all vertices of $A$.
    \item  $ |N_V(X)\setminus X |\geq |X|$ for all $X \subseteq A$.
\end{enumerate}
\end{thm}

\section{Proofs of Theorems \ref{thm:main:Meyniel} and \ref{thm:main:regimes}}\label{sec:Proofs:regime&meyniel}

We start by proving Theorem \ref{thm:main:regimes}, from which Theorem \ref{thm:main:Meyniel} follows as a direct consequence. 

\begin{proof}[Proof of Theorem \ref{thm:main:regimes}]
Let $G$ be a \smoothexp $k$-graph on $n$ vertices. Note that, as the stated bounds on the cop number are all clearly at least $20$, we can assume w.l.o.g. that $n\geq 20$. 

Our strategy is to choose the initial placement of our cops in such a way that we can catch the robber in $j$ moves for some $j \in \N$. In order to show the existence of such a choice of initial positions, we will use the expansion properties of $G$ to show that a random choice succeeds with positive probability.

In regimes (\ref{r:a}) and (\ref{r:d}) we will catch the robber by surrounding its $(j-1)$-st vertex-neigh\-bour\-hood. In fact, cops that start too far from the robber will not actively participate in the game. The following key claim characterises how many cops we need in the respective regimes to guarantee that sufficiently many cops are close enough to the starting vertex of the robber to make this strategy work. The proof of the claim is deferred to the end of this section.

\begin{claim}\label{claim:injection_vertexstrat}
Let $j \in \N$. There exists a subset $Y \subseteq V(G)$ such that for every vertex $v \in V(G)$ there exists an injection $f \colon \vNhd[j-1]{v}\to Y$ such that for every vertex $x \in \vNhd[j-1]{v}$, we have $\dist{x}{f(x)} \leq j$. Furthermore,
\begin{enumerate}
    \item[{\crtcrossreflabel{(a)}[i:a]}] if\/ $j \neq 1$ and $n^{\frac{1}{2j-1}} \leq d \leq \left(\frac{n}{k}\right)^{\frac{1}{2j-2}}$, then $|Y| \leq 20 \smoothconst^{-2}d^{j-1} \left\lceil \frac{n}{d^{2j-1}} \log n  \right\rceil$;
    \item[{\crtcrossreflabel{(d)}[i:d]}] if\/ $(nk)^{\frac{1}{2j}} \leq d \leq n^{\frac{1}{2j-1}}$, then $|Y| \leq 20 \smoothconst^{-1}\frac{n}{d^j} \log n$.
\end{enumerate}
\end{claim}

Given $Y$ as in the claim, the cops' strategy is to initially occupy the vertices of $Y$. The robber starts on some vertex $v$. 
By Claim \ref{claim:injection_vertexstrat} there exists an injection $f$ such that for every \emph{vertex} $x \in \vNhd[j-1]{v}$, the cop on vertex $f(x)$ is within distance $j$ of $x$ at the start of the game. Each cop which starts on a vertex $w$ in the image $f\left( \vNhd[j-1]{v}\right)$ moves to the vertex $f^{-1}(w)$ in the first $j$ moves. After $j-1$ moves the robber is positioned at some vertex $w \in \vNhd[j-1]{v}$. Since the cops move first, in the next turn the cop that started on the vertex $f(w)$ moves to the vertex $w$ and catches the robber. Note that in regime (\ref{r:a}), we use the described strategy after applying the index shift $j \to j+1$ to obtain the desired result.

In regimes (\ref{r:b}) and (\ref{r:c}) we instead catch the robber by surrounding his $j$-th edge-neighbourhood. Similar to the previous case, it suffices to prove the following claim, which we will do at the end of this section. 

\begin{claim}\label{claim:injection_edgestrat}
Let $j \in \N$. There exists a subset $Z \subseteq V(G)$ such that for all vertices $v \in V(G)$ there exists an injection $g \colon \eNhd[j]{v}\to Z$ , such that for every edge $e \in \eNhd[j]{v}$, we have $\dist{e}{g(e)} \leq j$. Furthermore,
\begin{enumerate}
    \item[{\crtcrossreflabel{(b)}[i:b]}] if\/ $\left(\frac{n}{k}\right)^{\frac{1}{2j}}\leq d \leq n^{\frac{1}{2j}}$, then $|Z| \leq 20 \smoothconst^{-1} \frac{n}{kd^j} \log n$;
    \item[{\crtcrossreflabel{(c)}[i:c]}] if\/ $n^{\frac{1}{2j}} \leq d \leq (nk)^{\frac{1}{2j}} $, then $|Z| \leq 20 \smoothconst^{-2}\frac{d^j}{k} \left\lceil \frac{n}{d^{2j}} \log n  \right\rceil \left\lceil \frac{k}{d^{j}} \log n  \right\rceil$.
\end{enumerate}
\end{claim}

Indeed, given such a set $Z$ the cops' strategy is to initially occupy the vertices of $Z$. The robber starts on some vertex $v$. By Claim \ref{claim:injection_edgestrat} there exists an injection $g$ such that for every \emph{edge} $e \in \eNhd[j]{v}$, we have $\dist{e}{g(e)} \leq j$. Each cop which starts on a vertex $w$ in the image $g\left(\eNhd[j]{v}\right)$ moves to some vertex $u$ in the edge $g^{-1}(w)$ in the first $j$ moves. After $j$ moves, the robber is on some vertex $x$ in some edge $e \in \eNhd[j]{v}$. Since the cops move first, the cop that started at $g(e)$ is currently at some vertex $u \in e$, and hence can catch the robber in the next move.
\end{proof}

It remains to prove the two claims.
\begin{proof}[Proof of Claim \ref{claim:injection_vertexstrat}]
Clearly such a set $Y$ exists if we do not make any restrictions on its size, so we may assume that one of \ref{i:a} or \ref{i:d} holds. We will choose our set $Y$ by specifying some probability $q$ and letting each vertex lie in $Y$ independently with probability $q$. We show that with positive probability a random choice of $Y$ satisfies the conclusions of the claim, and hence there must exist some suitable set $Y$.

Let us start with case \ref{i:d}, so we are assuming that $(nk)^{\frac{1}{2j}} \leq d \leq n^{\frac{1}{2j-1}}$. In this case we set $q = 10 \smoothconst^{-1} d^{-j} \log n$, where we note that our assumptions on $d$ ensure that $q \leq 1$. Since $|Y| \sim \text{Bin}(n,q)$, it follows from the Chernoff bound \eqref{thm:Chernoff:|x-Ex|} that $|Y| \leq 20  \frac{n}{\smoothconst d^{j}} \log n$ with probability at least $\frac{2}{3}$. We will show that an injection as stated in the claim exists with probability at least $\frac{2}{3}$, and hence $Y$ satisfies the conclusion of the claim with probability at least $\frac{1}{3} >0$.

Given a fixed vertex $v \in V$, an injection of the desired form corresponds to a matching in the bipartite graph $H=\left(\vNhd[j-1]{v} \cup Y, E\right)$ with edge set $E = \left\{(a,b) \colon a\in \vNhd[j-1]{v}, b \in Y, \dist{a}{b} \leq j\right\}$, which covers all vertices of $\vNhd[j-1]{v}$. 
 To find a matching as described above, it is enough to check that Hall's condition (see Theorem \ref{thm:Hall}) is satisfied for all $A \subseteq \vNhd[j-1]{v}$.

We split into two cases. First, suppose that $A \subseteq \vNhd[j-1]{v}$ with $a:=|A| \leq \frac{n}{d^j}$. Slightly abusing notation, we write $N_H(A)$ for the vertices at \emph{exactly} distance $1$ from $A$ in the auxiliary graph $H$. By Property \ref{Property:neighbourhood_estimates:vNhd} it follows that $|\vNhd[j]{A}| \geq \smoothconst ad^{j}$. Thus, by our choice of $Y$, we have $|N_H(A)|=|Y \cap \vNhd[j]{A}|$ stochastically dominates a binomial random variable $\Bin \left(\smoothconst ad^{j},q \right)$, the expectation of which satisfies $\expec{\Bin \left( \smoothconst ad^{j},q \right)}= \smoothconst ad^{j }q \geq 10 a \log n $.
Therefore, by the Chernoff bound \eqref{thm:Chernoff:10Ex} it follows that
\[
\prob{\,|N_H(A)|< |A|\,} \leq \prob{\Bin\left(\smoothconst ad^{j},q\right) < a} \leq \prob{\Bin\left(\smoothconst ad^{j},q\right) \leq a \log n} \leq \exp(-4 a \log n) = n^{-4a}.
\]
Then, using a union bound over all sets $A \subseteq \vNhd[j-1]{v}$ with $|A| \leq \frac{n}{d^j}$ we can bound the probability that there exists such a  set $A$ that violates Hall's condition from above by 
\begin{equation}\label{e:sumsmall}
\sum_{a =1}^{\frac{n}{d^j}} \binom{\left|\vNhd[j-1]{v}\right|}{a} n^{-4a} \leq \sum_{a =1}^{\frac{n}{d^j}} \binom{n}{a} n^{-4a} \leq \sum_{a = 1}^n n^{-3a} \leq \frac{1}{6n}.
\end{equation}

In the second case, when $a:=|A| > \frac{n}{d^j}$, we note that since $A \subseteq \vNhd[j-1]{v}$, by Property \ref{Property:neighbourhood_estimates:vNhd} we have $a \leq \left|\vNhd[j-1]{v}\right| \leq\smoothconst^{-1} d^{j-1}$ and $|\vNhd[j]{A}| \geq \smoothconst n$. Hence, $|N_H(A)|=|Y \cap \vNhd[j]{A}|$ stochastically dominates a binomial random variable $\Bin \left(\smoothconst n,q \right)$, the expectation of which satisfies 
\[
\expec{\Bin \left(\smoothconst n,q \right)}= \smoothconst n q = \frac{10n}{\smoothconst d^{j}} \log n \geq  10 \smoothconst^{-1} d^{j-1} \log n.
\]
Here we used the fact that in this regime we have $d \leq n^\frac{1}{2j-1}$. From the Chernoff bound \eqref{thm:Chernoff:10Ex} it follows that
\[
\prob{\,|N_H(A)|< |A|\,} \leq \prob{\Bin \left( \smoothconst n,q \right) <a } \leq \prob{\Bin \left(\smoothconst n,q \right) \leq \smoothconst^{-1} d^{j-1} \log n} \leq \exp(-4 d^{j-1} \log n) = n^{-4d^{j-1}}.
\]

Again, using the union bound and the fact that $a \leq \smoothconst^{-1} d^{j-1}$ we can bound the probability that there is such a set $A$ violating Hall's condition from above by
\begin{equation}\label{e:sumlarge}
 \sum_{a = \frac{n}{d^j}  +1}^{\smoothconst^{-1} d^{j-1}} \binom{\left|\vNhd[j-1]{v}\right|}{a} n^{-4d^{j-1}} \leq 2^{\left|\vNhd[j-1]{v}\right|} n^{-4d^{j-1}} \leq 2^{-\smoothconst^{-1} d^{j-1}} n^{-4d^{j-1}}\leq \frac{1}{6n}.
\end{equation}

Thus, for every vertex $v$ an injection of the desired form exists with probability at least $1-\frac{1}{3n}$. Using another union bound over all vertices, we can bound the probability that there exists a vertex for which there is no such injection by $\frac{1}{3}$, concluding the proof in the case \ref{i:d}.

In the case \ref{i:a}, where $n^{\frac{1}{2j-1}} \leq d \leq \left(\frac{n}{k}\right)^{\frac{1}{2j-2}}$, we proceed similarly as in case \ref{i:d}, but with a slightly different value of $q$. We let $q = 10 \smoothconst^{-2} \frac{d^{j-1}}{n} \left\lceil \frac{n}{d^{2j-1}} \log n \right\rceil$, noting again that our assumptions on $d$ ensure that $q \leq 1$.

Arguing as before, splitting into cases according to whether or not $|A| > \frac{n}{d^j}$, we see that it suffices to prove the following two inequalities:
\begin{equation}\label{Proof:regimeRefinement:firstsum}
\sum_{a = 1}^{\frac{n}{d^j}} \binom{\left|\vNhd[j-1]{v}\right|}{a} \,\prob{\Bin\left(\smoothconst ad^{j},q\right) < a } \leq \frac{1}{6n},  
\end{equation}
 and 
\begin{equation}\label{Proof:regimeRefinement:secondsum}
\sum_{a =\frac{n}{d^j}+ 1}^{\smoothconst^{-1} d^{j-1}} \binom{\left|\vNhd[j-1]{v}\right|}{a} \,\prob{\Bin \left(\smoothconst n,q \right) <  a } \leq \frac{1}{6n}.
\end{equation}  
To show \eqref{Proof:regimeRefinement:firstsum}, we note that 
\[
\expec{\Bin\left(\smoothconst ad^{j},q\right) } = \smoothconst ad^{j}q \geq 10 a \frac{d^{2j-1}}{n} \left\lceil \frac{n}{d^{2j-1}} \log n \right\rceil \geq 10 a \log n.
\]Then, as before, it is clear that \eqref{thm:Chernoff:10Ex} yields the desired concentration to bound the sum as in \eqref{e:sumsmall}.

Similarly, to show \eqref{Proof:regimeRefinement:secondsum}, we first note that, since $\left\lceil \frac{n}{d^{2j-1}} \log n \right\rceil \geq 1$, it follows that
$\expec{\Bin \left(\smoothconst n,q \right)} = \smoothconst n q \geq 10 \smoothconst^{-1} d^{j-1}$. Hence, by \eqref{thm:Chernoff:10Ex}
\[
\prob{\Bin \left(\smoothconst n,q \right) > \smoothconst^{-1} d^{j-1} } \leq e^{ -4 \smoothconst^{-1} d^{j-1} },
\]
and since $a \leq \smoothconst^{-1}d^{j-1}$ by Property \ref{Property:neighbourhood_estimates:vNhd}, we can bound the sum as in \eqref{e:sumlarge}, although we have to be more careful in our estimates. 
Plugging into \eqref{Proof:regimeRefinement:secondsum} we obtain
\begin{align}
\sum_{a =\frac{n}{d^j}+ 1}^{\smoothconst^{-1} d^{j-1}} \binom{\left|\vNhd[j-1]{v}\right|}{a} e^{ -4 \smoothconst^{-1} d^{j-1} } &\leq 2^{ \smoothconst^{-1} d^{j-1}} e^{-4 \smoothconst^{-1} d^{j-1}} \label{Proof:regimeRefinement:thirdsum}  \\ 
&\leq e^{-3 \smoothconst^{-1} d^{j-1}} \leq e^{-3 n^\frac{1}{3}} \leq \frac{1}{6n}, \notag
\end{align}where we used the facts that $d^{j-1} \geq n^\frac{j-1}{2j-1} \geq n^\frac{1}{3}$, as $j \geq 2$, and that $n \geq 20$. 
\end{proof}

\begin{proof}[Proof of Claim \ref{claim:injection_edgestrat}]
As in the previous claim, the existence of such a set $Y$ is clear if we make no assumptions on its size, so we may assume that one of \ref{i:b} or \ref{i:c} holds. Again, we will choose the set $Y$ by letting each vertex lie in $Y$ independently with some fixed probability $q$, and show that with positive probability such a set $Y$ satisfies the conclusions of the claim. 

Let us start with case \ref{i:c}, where $n^\frac{1}{2j} \leq d \leq (nk)^\frac{1}{2j}$. Here we set $q=10\smoothconst^{-2}  \frac{d^{j}}{ n k}  \left\lceil \frac{n}{d^{2j}} \log n  \right\rceil \left\lceil \frac{n}{d^{2j}} \log n  \right\rceil$. As in the previous claim, it follows from \eqref{thm:Chernoff:|x-Ex|} that $|Y| \leq 20  \frac{d^{j}}{\smoothconst^{2}k} \left\lceil \frac{n}{d^{2j}} \log n  \right\rceil \left\lceil \frac{n}{d^{2j}} \log n  \right\rceil$ with probability at least $\frac{2}{3}$. We will show that an injection as stated in the claim exists with probability at least $\frac{2}{3}$, and hence $Y$ satisfies the conclusion of the claim with probability at least $\frac{1}{3} >0$.

Given a fixed $v \in V$, an injection of the desired form corresponds to a matching in the bipartite graph $H= \left(\eNhd[j]{v} \cup Y,F\right)$ with $F = \left\{(e,b) \colon e\in \eNhd[j]{v}, b \in Y, \dist{e}{b} \leq j\right\}$, which covers all of $\eNhd[j]{v}$. To find such a matching, it is enough to check that Hall's condition (see Theorem \ref{thm:Hall}) is satisfied for all $B \subseteq \eNhd[j]{v}$. We again split into two cases, depending on the size of $B$.

First, suppose that $b:=|B| \leq \frac{n}{kd^j}$. Then, by Property \ref{Property:neighbourhood_estimates:vNhd:edges} it follows that $\left|\vNhd[j]{B}\right| \geq \smoothconst b k d^{j}$ and so $|N_H(B)| = \left|Y \cap \vNhd[j]{B}\right|$ stochastically dominates a binomial random variable $\Bin \left(\smoothconst b k d^{j},q \right)$ with expectation 
\[
\expec{\Bin \left(\smoothconst b k d^{j},q \right)}=\smoothconst b k d^{j} q \geq 10 b \frac{d^{2j}}{n} \left\lceil \frac{n}{d^{2j}} \log n  \right\rceil \geq  10 b \log n.
\]

Similarly to the previous case, using the Chernoff bound \eqref{thm:Chernoff:10Ex} we can bound the probability that Hall's condition fails for such a set $B$ from above by
\[
\prob{|N_H(B)| < b} \leq \prob{\Bin \left(\smoothconst b k d^{j},q \right) < b} \leq n^{-4b}.
\]
Taking a union bound over all sets $B\subseteq \eNhd[j]{v}$ with $|B| \leq \frac{n}{kd^j}$ as in \eqref{e:sumsmall}, we see that the probability that any such set violates Hall's condition is at most $ \frac{1}{6n}$. 

In the second case when $b:=|B|> \frac{n}{kd^j}$, we note that as $B \subseteq \eNhd[j]{v}$, by Property \ref{Property:neighbourhood_estimates:eNhd} we have $b \leq \frac{d^j}{\smoothconst k}$ and by Property \ref{Property:neighbourhood_estimates:vNhd:edges} we have $|\vNhd[j]{B}| \geq \smoothconst n$. Hence, $|N_H(B)|=\left|Y \cap \vNhd[j]{B}\right|$ stochastically dominates a binomial random variable $\Bin \left(\smoothconst n,q \right)$ with expectation $\expec{\Bin \left(\smoothconst n,q \right)}= \smoothconst n q \geq 10 \frac{d^j}{\smoothconst k}$. Hence, from \eqref{thm:Chernoff:10Ex} it follows that
\[
\prob{\Bin \left( \smoothconst n,q \right) <b } \leq \prob{\Bin \left(\smoothconst n,q \right) \leq  \frac{d^j}{\smoothconst k}} \leq e^{-4  d^{j}/(\smoothconst k) }.
\]

Again, taking a union bound over all sets $B$ of size $\frac{n}{kd^j} < |B| \leq \frac{d^j}{\smoothconst k}$ we can bound from above the probability that Hall's condition fails for some $B$ with $|B| k d^j > n$ from above as in \eqref{Proof:regimeRefinement:thirdsum}. Note that in the case $j =1$ and $k \geq \frac{\dhat}{\log n}$ we need the additional factor $\left\lceil \frac{n}{d^{2j}} \log n  \right\rceil$ for this union bound to work.  

Hence, for every vertex $v$, an injection of the desired form exists with probability at least $1-\frac{1}{3n}$. Using another union bound over all vertices, we can bound the probability that there exists a vertex for which there is no such injection by $\frac{1}{3}$, concluding the proof in the case \ref{i:c}.

The proof in the case \ref{i:b} is again analogous, using $q= \frac{10}{\smoothconst kd^j} \log n$. Since the calculations are similar in nature to cases \ref{i:a}, \ref{i:c} and \ref{i:d} we omit them.
\end{proof}
Theorem \ref{thm:main:Meyniel} then follows by checking that in each regime the bounds given by Theorem \ref{thm:main:regimes} are not significantly larger than $\sqrt{\frac{n}{k}}$.

\begin{proof}[Proof of Theorem \ref{thm:main:Meyniel}]
Let $j \in \N$ be such that the average degree $d:=d(G)$ of $G$ falls into one of the regimes defined in Theorem \ref{thm:main:regimes}. We note that the upper bound in regimes (\ref{r:a}) and (\ref{r:c}) is increasing in $d$ and the upper bound in regimes (\ref{r:b}) and (\ref{r:d}) is decreasing in $d$. Hence, the bound obtained by applying Theorem \ref{thm:main:regimes} to $G$ is at least as good as the bound given at the beginning of regime (\ref{r:b}), where $d = \left(\frac{n}{k}\right)^\frac{1}{2j}$, or at the beginning of regime (\ref{r:d}), where $d = (nk)^{\frac{1}{2j}}$.

In the first case, the cop number is bounded from above by 
\[
\cop{G} \leq  20 \smoothconst^{-2} \frac{n}{k d^j}\log n = 20 \smoothconst^{-2} \frac{n}{k }\sqrt{\frac{k}{n}}\log n= 20 \smoothconst^{-2} \sqrt{\frac{n}{k}} \log n,
\]
and in the second case by 
\[
\cop{G} \leq 20 \smoothconst^{-2} \frac{n}{d^j}\log n= 20 \smoothconst^{-2} \frac{n}{\sqrt{nk}}\log n = 20 \smoothconst^{-2} \sqrt{\frac{n}{k}} \log n.
\]
\end{proof}

\section{Proof of Theorem \ref{thm:main:Gknp}}\label{sec:proof:Gknp}

To show the existence of a $\smoothconst \in (0,1]$ such that whp $\Gknp$ is \smoothexp we will proceed as follows: First, we will show that $\Gknp$ expands very well (in terms of the expansion constant) in the first edge-neighbourhood and in the first vertex-neighbourhood. Then we will inductively use these results to extend the expansion properties to larger distance neighbourhoods. As might be expected, the expansion constant remains quite good until the neighbourhoods come close to containing the whole graph. We begin by showing that the first edge-neighbourhoods of subsets of $\Gknp$ expand well. In fact, for our inductive step it will be necessary to show a stronger property that unless a subset $A$ is too large, the size of its first edge-neighbourhood will be quite tightly concentrated around the value $\frac{\hat{d}}{k} |A|$, where we recall from \eqref{eq:average_degree_hat} that $\dhat = pk\binom{n-1}{k-1}$. In order to get good estimates in larger neighbourhoods, we need to calculate the deviation in the first neighbourhood quite precisely. To this end, we set 
\begin{align*}
    \estimerror \coloneqq \frac{\sqrt{\log \log n}}{\log n}.
\end{align*}
\begin{lem}\label{lemma:edge_neighbourhood_estimate}
Assume the parameters of $G= \Gknp$ are such that $\frac{\dhat}{k} = \omega \left(\log^3 n \right)$ and $k \leq \frac{n}{4}$.  
Then whp the following holds:

For every subset $A \subseteq [n]$ satisfying $|A|  \leq \frac{2 n}{k \log n}$, 
\begin{equation}\label{eq:edge_neighbourhood_estimate:small}
 \left(1-\estimerror\right)|A|\frac{\dhat}{k} \leq |\eNhd[1]{A}| \leq  \left(1+\estimerror \right)|A|\frac{\dhat}{k}.
\end{equation}

Moreover, for every subset $A \subseteq [n]$,
\begin{equation}\label{eq:edge_neighbourhood_estimate:large}
|\eNhd[1]{A}| \geq \frac{1}{16} \minp{}{|A|\frac{\dhat}{k},\frac{n}{k}}.   
\end{equation}
\end{lem}

\begin{proof}
 We will start by proving the first statement. Let $A \subseteq [n]$ be given where $a:= |A| \leq \frac{2 n}{k \log n}$ 
 and let $X = |\eNhd[1]{A}|$ denote the number of edges in $G$ that contain at least one vertex of $A$. 
Let us write $\edgecompleteset$ for the number of edges meeting $A$ in the complete $k$-uniform hypergraph $K^k(n)$ on $n$ vertices. 
 A standard inclusion-exclusion type argument implies that 
\[
\edgecompleteset= \sum_{j =1}^{\minp{}{a,k}} \binom{a}{j} \binom{n-j}{k-j} (-1)^{j+1}.\]
The (absolute) ratio of consecutive terms in the sum is given as

\begin{equation}\label{eq:Proof:Lemma:edge_neighbourhood_estimate}
 \frac{\binom{a}{j} \binom{n-j}{k-j}}{\binom{a}{j-1} \binom{n-j+1}{k-j+1}} = \frac{(a-j+1) (k-j+1)}{ j(n-j+1)} < \frac{a k}{n} \leq \frac{1}{2}, \quad \text{for all } 2 \leq j \leq \minp{}{a,k}.
\end{equation}
Hence, $\edgecompleteset$ is dominated by the first term of the sum and the total contribution from all latter terms is at most an $O\left(\frac{ak}{n} \right)$-fraction of this value. More formally, 
\[
M = \sum_{j =1}^{\min\{a,k\}} \binom{a}{j} \binom{n-j}{k-j} (-1)^{j+1} = 
a \binom{n-1}{k-1} \left(1 + O\left(\frac{ak}{n} \right)\right).
\]
Since $X \sim \text{Bin}(M,p)$, it follows from $\dhat \coloneqq pk \binom{n-1}{k-1}$ (see \eqref{eq:average_degree_hat}) and the fact that $\frac{ak}{n}= O\left(\frac{1}{ \log n}\right) = o(\delta)$ that
\[
\expec{X} = Mp = p a \binom{n-1}{k-1} \left(1 + O\left( \frac{1}{\log n}\right)\right) = \frac{a \dhat}{k}\left(1 + o(\delta)\right).
\]
Hence, by the Chernoff bound \eqref{thm:Chernoff:|x-Ex|} it follows that 
\[
\prob{\left|X-\frac{a\dhat}{k}\right|> \frac{a\dhat \estimerror }{k}} \leq 2 \exp \left(-\frac{a \dhat \estimerror^2}{3k} \right).
\]

Therefore, by a union bound, the probability that there exists a set $A \subseteq [n]$ with $|A| \leq \frac{2 n}{k \log n}$ such that $|\eNhd[1]{A}|$ differs from $|A|\frac{\dhat}{k}$ by more than $\frac{|A|\dhat \estimerror}{k}$ is at most
\[
\sum_{a = 1}^{\frac{2 n}{k \log n}}  2 \binom{n}{a} \exp \left(-\frac{ a \dhat \estimerror^2}{3k} \right) \leq \sum_{a =  1}^\infty 2 n^a n^{-a \omega (1)} = o(1),
\]
where we used our assumption that $\frac{\dhat}{k } = \omega \left(\log^3 n\right) = \omega \left(\estimerror^{-2} \log n \right)$.
Therefore, whp 

\[
 \left(1-\estimerror\right)\frac{a \dhat}{k} \leq |\eNhd[1]{A}| \leq  \left(1+\estimerror \right)\frac{a\dhat}{k}.
 \]

To show the second statement we split into two cases. Firstly, let us assume that $a:=|A| \leq \frac{n}{2k}$. 
Let $X$ and $\edgecompleteset$ be defined as above and note that by our assumption on the size of $A$ \eqref{eq:Proof:Lemma:edge_neighbourhood_estimate} still holds. 
Hence, $\edgecompleteset$ is an alternating sum with decreasing terms, and therefore is bounded from below by the difference of the first two terms, and it follows from \eqref{eq:Proof:Lemma:edge_neighbourhood_estimate} that
\[
M \geq \frac{a}{2}\binom{n-1}{k-1}.
\]
Again, since $X \sim \text{Bin}(M,p)$ and using $d \coloneqq pk\binom{n-1}{k-1}$ we have
\begin{equation*}
\expec{X} = M p \geq  p \frac{a}{2} \binom{n-1}{k-1} = \frac{a\dhat}{2k},
\end{equation*}
and hence, by the Chernoff bound \eqref{thm:Chernoff:onesided},
\[
\prob{X \leq \frac{a\dhat}{4k}} \leq \prob{X \leq  \frac{1}{2} \expec{X}} \leq \exp\left(-\frac{\expec{X}}{8}\right) \leq \exp \left(-\frac{a \dhat}{16 k}\right).
\]
Thus, by a union bound, we can bound the probability that there exists a set $A\subseteq [n]$ with $|A| \leq \frac{n}{2k}$ and $|\eNhd[1]{A}| \leq \frac{a\dhat}{4k}$ from above by
\[
\sum_{a = 1}^{\frac{n}{2k}} \binom{n}{a} \exp \left(-\frac{a \dhat }{16 k } \right) \leq \sum_{a= 1}^\infty n^a n^{-a \omega(1)} =  o(1),
\]
where we used our assumption that $\frac{\dhat}{k} = \omega(\log n)$. 
Therefore, whp for every set $A \subseteq [n]$ with $|A| \leq \frac{n}{2k}$,
\[
|\eNhd[1]{A}| \geq\frac{a\dhat}{4k}  .
\]

If $A$ is such that $ a > \frac{n}{2k}$ we pick the largest possible subset $A' \subseteq A$ such that $a'k:=|A'| k \leq n/2 $. Therefore, $( a'+1)k > n/2$, and so $a'k \geq n/2 - k > n/4$.  
Using the previous argument and the fact that $\dhat \geq k$ we see that 
\[
|\eNhd[1]{A}| \geq \left|\eNhd[1]{A'}\right| \geq \frac{a'\dhat}{4 k} \geq \frac{a'k}{4k} \geq \frac{n}{16k}.
\]

In total, the previous two cases yield that whp for all subsets $A\subseteq [n]$,
\[
|\eNhd[1]{A}| \geq  \minp{}{\frac{a\dhat}{4k},\frac{n}{16k}} \geq \frac{1}{16} \minp{}{\frac{a\dhat}{k},\frac{n}{k}}.
\]

\end{proof}
To show an analogous result for the vertex-neighbourhood, we recall that for a set of edges $B$ we defined $\vine{B} = \{v \in e \colon e \in B\}$ and that the vertex-neighbourhood $\vNhd[1]{A}$ of a set $A$ can be written as $\vine{\eNhd[1]{A}}$. Thus, to show that the first vertex-neighbourhood expands well, it will suffices to show that in $\Gknp$ sets of edges are unlikely to have large overlaps.
\begin{lem}\label{lemma:edge_set_overlap}
Assume the parameters of $G = \Gknp$ are such that $ k = \omega(\log n) $, $ k \leq 2^{-11} n$ and $\dhat \leq n$.  
Then whp the following holds. 

For every subset $B\subseteq E(G)$ of edges and every $\epsilon = \epsilon(n) \in \left(0,\frac{1}{2}\right]$ such that $\left(\frac{|B|k}{n}\right)^\epsilon \leq 2^{-5}$, 
\begin{equation}\label{eq:edge_set_overlap:small}
| \vine{B} | \geq \left(1-\epsilon \right) |B| k .  
\end{equation}
Moreover, for every subset $B\subseteq E(G)$ of edges,
\begin{equation}\label{eq:edge_set_overlap:large}
|\vine{B}| \geq 2^{-12} \min \{|B| k, n\}.   
\end{equation}

\end{lem}
\begin{proof} 
We start by showing the first statement. Let $b=b(n) \in \mathbb{N}$ and $\epsilon = \epsilon(n) \in \left(0,\frac{1}{2}\right]$ be such that $\left(\frac{bk}{n}\right)^\epsilon \leq 2^{-5}$. We set $\tb \coloneqq bk(1-\epsilon)$ and denote by $X_b$ the number of edge sets $B$ in $G$ with $|B|=b$ and $|V_B| \leq \tb$. 

To bound the expectation of $X_b$, we count the number of possible choices for such an edge set $B$ as follows: We think of $B$ as a partition of the multi-set $\hat{V}_B$, which is given by including each element $x$ of $V_B$ with multiplicity equal to the number of edges in $B$ which contain $x$, into $k$-sets. In particular, each such edge set $B$ can be specified by fixing the set $V_B$ of size $t \leq \tb$, the vector $(x_1,\ldots, x_t)$ determining the multiplicity of each vertex in $B$, and a partition of the multi-set $\hat{V}_B$ into $b$ many $k$-sets.

Now, there are at most $\binom{n}{t}$ many possible sets $V_B$ and, since $\sum_{i=1}^t x_i = bk$, it follows that there are at most $\binom{bk+t-1}{t}$ ways to choose the vector $(x_1,\ldots,x_t)$. Finally, a crude upper bound for the number of partitions is $b^{bk}$, since each of the $bk$ vertices in $\hat{V}_B$ has to be assigned to one of the $b$ many $k$-sets.

It follows that
\begin{equation}\label{e:Xexp}
\expec{X_b} \leq \sum_{t=1}^{tb} \binom{n}{t}\binom{bk+t-1}{t} b^{bk} p^b.
\end{equation}
Using our assumption that $\hat{d} \leq n$ we get that
\[
p \leq \frac{n}{k \binom{n-1}{k-1}} \leq \frac{n}{k \left(\frac{n-1}{k-1}\right)^{k-1}} \leq \left(\frac{k}{n}\right)^{k-2},
\]
and hence, using the fact that $\binom{n}{t}$ is increasing for $1 \leq t \leq \tb$, we can bound $\expec{X_b}$ from above by
\begin{align*}
\expec{X_b} &\leq \sum_{t = 1}^{\tb} \binom{n}{t} \binom{bk +t-1}{t} b^{bk} p^b\\
&\leq t_b \left( \frac{en}{t_b} \right)^{t_b} 2^{2bk} b^{bk} \left(\frac{k}{n}\right)^{b(k-2)}\\
&= \left(t_b^{\frac{1}{bk}} \left( \frac{n}{k} \right)^{\frac{2}{k}} \frac{4e^{1-\epsilon}}{(1-\epsilon)^{1-\epsilon} }\left(\frac{bk}{n}\right)^{\epsilon} \right)^{bk}.\\
\end{align*}
However, since $k=\omega(\log n)$ and $bk \leq n$, it follows that $t_b^{\frac{1}{bk}} \left( \frac{n}{k} \right)^{\frac{2}{k}} \leq 2$. Furthermore, it is easy to check that $(1-\epsilon)^{(1-\epsilon)}$ is decreasing on $\left(0,\frac{1}{2}\right]$ and hence
\begin{align*}
\expec{X_b} &\leq \left(\frac{8e}{\sqrt{2} }\left(\frac{bk}{n}\right)^{\epsilon} \right)^{bk} \leq \left(2^4 \left(\frac{bk}{n}\right)^{\epsilon} \right)^{bk} = o\left (\frac{1}{n}\right),
\end{align*}
where we used our assumption $2^4\left(\frac{bk}{n}\right)^{\epsilon}  \leq \frac{1}{2}$ and $bk = \omega( \log n)$. In particular, since $b \leq 2^{-\frac{5}{\epsilon}} \frac{n}{k} \leq n$, we can conclude by a union bound over all possible values of $b$ that whp there are no edge sets violating the first part of the lemma.

To show the second statement, suppose that $B \subseteq E(G)$ is given. If $|B| \leq \frac{n}{2^{10}k}$, then 
\[
\left( \frac{|B|k}{n}\right)^{\frac{1}{2}} \leq 2^{-5},
\]
and so by the first part of the lemma with $\epsilon=\frac{1}{2}$, we have
\[
|\vine{B}| \geq \frac{1}{2} |B|k \geq 2^{-12} \min \{ |B|k,n\}.
\]

If $|B|  > \frac{n}{2^{10}k}$ we simply pick a largest subset $B' \subset B$ such that $b' := |B'| \leq \frac{n}{2^{10}k}$. Then, $b' >  \frac{n}{2^{10}k} - 1 \geq \frac{n}{2^{11}k}$, for large enough $n$.
By the previous observation, it follows that 
\[
| \vine{B} | \geq | V_{B'} | \geq b'k/2 \geq \frac{n}{2^{12}} \geq 2^{-12} \min \{ |B|k,n\},
\]
proving also the second statement.
\end{proof} 

We note that an immediate corollary of Lemmas \ref{lemma:edge_neighbourhood_estimate} and \ref{lemma:edge_set_overlap} is that not too large sets in $\Gknp$ have relatively uniform vertex expansion.

\begin{lem}\label{l:vertex_expansion}
Assume the parameters of $G = \Gknp$ are such that $\frac{\hat{d}}{k} = \omega(\log^3 n)$, $k=\omega(\log n)$ and $\hat{d}\leq n$. Then whp the following holds.
For every subset $A \subseteq V(G)$ and every $\epsilon=\epsilon(n) \in \left(0,\frac{1}{2}\right]$ such that $\left(\frac{|A| \hat{d}}{n}\right)^{\epsilon} \leq 2^{-6}$,
\begin{equation}\label{e:vertex_neighbourhood_estimate}
(1-\epsilon)\left(1-\estimerror \right) |A| \hat{d} \leq |N_V(A)| \leq \left(1+\estimerror \right) |A| \hat{d}.
\end{equation}
Moreover, for every subset $A \subseteq [n]$,
\begin{equation}\label{eq:vertex_neighbourhood_estimate:large}
2^{-16} \minp{}{|A| \hat{d},n} \leq  |N_V(A)| \leq 2^{12} |A| \hat{d}.
\end{equation}
\end{lem}
\begin{proof}
We start by showing the first statement. Given a set $A$ and $\epsilon$ satisfying the conditions of the corollary, we note that, since $\hat{d} = \omega \left( k \log^3 n\right)$ and $\epsilon \leq \frac{1}{2}$,
\[
|A| \leq 2^{-\frac{6}{\epsilon}} \frac{n}{\hat{d}} = o\left(\frac{n}{k\log n} \right).
\]
Hence we can apply the first part of Lemma \ref{lemma:edge_neighbourhood_estimate}  to conclude that whp
\begin{equation}\label{e:intermediarystep}
 \left(1-\estimerror\right)|A|\frac{\dhat}{k} \leq |\eNhd[1]{A}| \leq  \left(1+\estimerror \right)|A|\frac{\dhat}{k}.
\end{equation}

If we let $B = |\eNhd[1]{A}|$ then it is immediate that
\[
|N_V(A)| = |V_B| \leq k|B| \leq \left(1+\estimerror \right)|A| \hat{d}.
\]
On the other hand, by \eqref{e:intermediarystep} and our assumption on $|A|$
\[
\left(\frac{|B|k}{n}\right)^{\epsilon} \leq \left( \left(1+\estimerror\right)|A|\frac{\dhat}{n}\right)^{\epsilon} \leq 2 \left(|A|\frac{\dhat}{n}\right)^{\epsilon} \leq 2^{-5}.
\]
Hence we can apply Lemma \ref{lemma:edge_set_overlap} to conclude that
\[
|N_V(A)| = |V_B| \geq (1-\epsilon)|B|k \geq (1-\epsilon) \left(1-\estimerror \right)|A| \hat{d}
\]
as claimed.

To show the second statement, let $A \subseteq [n]$ and assume first that $|A| \dhat \leq 2^{-12}n.$ Then, by the first statement with $\epsilon = \frac{1}{2}$,
\begin{align*}
  \frac{1}{4} |A| \hat{d} \leq \frac{1}{2} \left(1-\estimerror \right) |A| \hat{d} \leq |N_V(A)| \leq \left(1+\estimerror \right) |A| \hat{d} \leq 2 |A| \hat{d},
\end{align*}showing the second statement in this case. 
On the other hand, for $|A| \dhat > 2^{-12}n$, note first the trivial upper bound
\begin{align*}
    |\vNhd[1]{A}| \leq n \leq 2^{12} |A| \dhat.
\end{align*}
For the lower bound, we can apply the second part of Lemma \ref{lemma:edge_neighbourhood_estimate} to conclude that 
\begin{align*}
|\eNhd[1]{A}| \geq 2^{-4} \minp{}{|A| \frac{\dhat}{k},\frac{n}{k}}.
\end{align*}
Setting $B = \eNhd[1]{A}$ it now follows from \Cref{lemma:edge_set_overlap} that
\begin{align*}
    |N_V(A)| = |V_B| \geq 2^{-12} \minp{}{|B|k,n} \geq 2^{-16} \minp{}{|A| \dhat,n},
\end{align*} concluding the proof. 
\end{proof}

Note that, in particular, Lemma \ref{l:vertex_expansion} implies that $|\vNhd[1]{v}|$ is roughly $\dhat$ for every vertex $v$ of $\Gknp$. To prove Theorem \ref{thm:main:Gknp} we will first show that neighbourhoods in $\Gknp$ expand well in terms of $\dhat$ and then finally conclude by showing that $d$ is whp close to $\dhat$. 

\begin{proof}[Proof of Theorem  \ref{thm:main:Gknp}]
We note first that by our assumptions on $k$ and $p$, $\Gknp$ satisfies the conditions of Lemmas \ref{lemma:edge_neighbourhood_estimate}, \ref{lemma:edge_set_overlap} and \ref{l:vertex_expansion}. We therefore assume in what follows that the conclusions of these lemmas hold deterministically in $\Gknp$.

For $i \in \{1,2,3\}$, we say a graph $G$ satisfies Property $\mathbf{(A. i)}$', if there exists a constant $0 < \xi'_i \leq 1$ such that $G$ satisfies Property $\mathbf{(A.i)}$ for this constant, but when $d$ is replaced by $\dhat$ in \Cref{Property:neighbourhood_estimates}. We will start by showing that $\Gknp$ satisfies the Properties $\mathbf{(A.i)}$'. Taking a minimum over all respective constants, we obtain a universal constant $\xi'$ such that $G$ satisfies all properties $\mathbf{(A.i)}$' for this constant. We conclude by showing that $\dhat$ lies sufficiently close to $d$ and in particular that properties $\mathbf{(A.i)}$' imply properties $\mathbf{(A.i)}$ for a universal constant $\xi$ that is only a constant factor smaller than $\xi'$.

Our first step will be to show that the size of the vertex-neighbourhoods in $\Gknp$ grow relatively uniformly for small enough sets, by inductively applying Lemma \ref{l:vertex_expansion}. However, we will need to carefully pick the parameter $\epsilon$ in each step so that the cumulative error in these approximations is not too large.

Let $A \subseteq [n]$ with $a := |A|$ and let $r \in \mathbb{N}$ be such that $a\hat{d}^r \leq \frac{n}{2 \log n}$. Note that, since $\hat{d} = \omega\left( \log^4 n\right)$, it follows that $r \leq \frac{\log n}{4 \log \log n}$. Our aim will be to show the following
\begin{equation}\label{e:Aneighbourhoodbound}
2^{-5} a \dhat^{r} \leq  \left|\vNhd[r]{A}\right|   \leq 2a \dhat^{r}.
\end{equation}
In particular, note that if we take $A=\{v\}$ to be a single vertex then, since $\sqrt{nk} \leq \sqrt{n \frac{\dhat}{\log^3 n}} \leq \frac{n}{2 \log n}$, it follows from \eqref{e:Aneighbourhoodbound} that Property \ref{Property:neighbourhood_estimates:eNhd}' holds for $v$ with $\xi_1' = 2$.

In order to apply Lemma \ref{l:vertex_expansion}, let us set $\epsilon_0 = 0$ and
\begin{equation}\label{e:epsilondef}
\epsilon_i \coloneqq \frac{5}{\log n - \log \left(2 a\dhat^{i}\right)} \quad \text{for } 1\leq i \leq r.
\end{equation}
Note that, since $a\dhat^{i} \leq \frac{n}{2 \log n}$, it follows that $\epsilon_i = o(1)$ for each $i \leq r$.

We claim inductively that the following bound holds for each $0\leq i \leq r$:
\begin{equation}\label{Proof:Prop_estimates:induction:vNhd}
\prod_{j=0}^{i}\left(1-\epsilon_j \right) \left(1-\estimerror \right)^i a\dhat^i \leq |\vNhd[i]{A}| \leq   \left(1+\estimerror \right)^i  a\dhat^i,
\end{equation}
where the statement is clear for $i=0$.

Suppose that \eqref{Proof:Prop_estimates:induction:vNhd} holds for some $i <r$. Then, since $i < r = o\left( \frac{1}{\estimerror}\right)$, we have
\begin{equation}\label{e:loose_bounds}
\frac{1}{2}\prod_{j=0}^{i}\left(1-\epsilon_j \right) a \dhat^{i} \leq \prod_{j=0}^{i}\left(1-\epsilon_j \right) \left(1-\estimerror \right)^{i} a\dhat^{i} \leq \left|\vNhd[i]{A}\right|  \leq \left(1+ \estimerror \right)^{i}  a\dhat^{i} \leq 2 a\dhat^{i}
\end{equation}
and so
\[
\left(\frac{\left|\vNhd[i]{A}\right| \dhat}{n}\right)^{\epsilon_{i+1}} \leq \left(\frac{2a\dhat^{i+1}}{n}\right)^{ \frac{5}{\log n -  \log \left(2 a\dhat^{i+1}\right)} } \leq 2^{-6}.
\]
Hence we can apply Lemma \ref{l:vertex_expansion} to $\vNhd[i]{A}$ to conclude that
\[
\prod_{j=0}^{i+1}\left(1-\epsilon_j \right) \left(1-\estimerror \right)^{i+1} a\dhat^{i+1} \leq \left|N_V\left(\vNhd[i]{A}\right)\right| = \left|\vNhd[i+1]{A}\right|   \leq \left(1+\estimerror \right)^{i+1}  a\dhat^{i+1},
\]
and so the induction step holds.

In particular, taking $i=r$, it follows from \eqref{e:loose_bounds} that
\begin{equation}\label{e:almostA1}
\frac{1}{2} \prod_{j=0}^{r}\left(1-\epsilon_j \right)  \dhat^{r} \leq  \left|\vNhd[r]{A}\right|   \leq 2\dhat^{r}.
\end{equation}
Hence it remains to bound the term $\prod_{j=0}^{r}\left(1-\epsilon_j \right)$. The following claim, whose verification we defer to the end of the proof, provides such a bound.

\begin{claim}\label{claim:inequalities:proof:gknp}
Let $a \leq n$ and $r$ be such that $a\dhat^r \leq \frac{n}{2 \log n}$, and let $\epsilon_i$ be defined as in \eqref{e:epsilondef}. Then
\begin{align}
\prod_{j=0}^{r}\left(1-\epsilon_j \right)  \geq 2^{-4}. \label{Proof:Prop_estimates:epsilon_inequality}
\end{align}  
\end{claim}

It is now clear that \eqref{e:almostA1} and Claim \ref{claim:inequalities:proof:gknp} together imply \eqref{e:Aneighbourhoodbound}.

Now let us turn to Property \ref{Property:neighbourhood_estimates:vNhd}'. Let us fix a subset $A \subseteq [n]$ with $a:=|A|$ and $r \in \mathbb{N}$. We let $r_0 = \min \Bigl\{ r, \maxp{}{i : |A|\dhat^i \leq \frac{n}{2 \log n}}\Bigr\}$ and let $A' = N_V^{r_0}(A)$. Then, by \eqref{e:Aneighbourhoodbound}, we have
\begin{equation}\label{e:A'bound}
 2^{-5} a \dhat^{r_0} \leq  \left|A'\right|  \leq 2a \dhat^{r_0}.
\end{equation}
In particular, if $r=r_0$, then \eqref{e:A'bound} implies that \ref{Property:neighbourhood_estimates:vNhd}' holds for $A$ with $\xi_2'= 2^{-5}$.

If $r = r_0+1$, then by \eqref{e:A'bound} and the second part of \Cref{l:vertex_expansion} 
\begin{align}\label{e:A'r0+1bound}
    2^{-21} \minp{}{ad^r,n} \leq |\vNhd[1]{A'}| = |\vNhd[r]{A}| \leq 2^{13} ad^r,
\end{align}
implying that Property \ref{Property:neighbourhood_estimates:vNhd}' holds for $A$ with $\smoothconst_2' = 2^{-21}$. Finally, for $r \geq r_0+2$ note that since $\dhat = \omega(\log^3 n)$, we have $a\dhat^r \geq n$, and so applying the second part of \Cref{l:vertex_expansion} to $\vNhd[1]{A'}$ yields together with \eqref{e:A'r0+1bound}
\begin{align*}
   2^{-37}n = 2^{-37} \minp{}{ad^r,n} = 2^{-37} \minp{}{ad^{r_0+2},n} \leq |\vNhd[2]{A'}| \leq |\vNhd[r]{A}| \leq n \leq ad^r,
\end{align*} hence Property \ref{Property:neighbourhood_estimates:vNhd}' holds also in this case with $\smoothconst_2' = 2^{-37}$.

Let us finally turn to Property \ref{Property:neighbourhood_estimates:vNhd:edges}'. Let $B \subseteq E(G)$. By the second part of Lemma \ref{lemma:edge_set_overlap}, we conclude that
\[
|\vine{B}| \geq 2^{-12} \minp{}{|B|k,n}.
\]
However, since $N_V^r(B) = N_V^r(V_B)$, we can apply \ref{Property:neighbourhood_estimates:vNhd}' to $V_B$ to conclude that
\begin{align*}
   |\vNhd[r]{B}| &=  |\vNhd[r]{\vine{B}}| \geq 2^{-37} \minp{}{|\vine{B}|\dhat^r,n} \\
   &\geq 2^{-37} \minp{}{2^{-12}\minp{}{|B|k,n}\dhat^r,n} \\
   &\geq 2^{-49} \minp{}{|B|k\dhat^r,n }, 
\end{align*}
and hence $\Gknp$ satisfies Property \ref{Property:neighbourhood_estimates:vNhd:edges}' with $\xi'_3=2^{-49}$.
It remains to show that Properties $\mathbf{(A.i)}$' imply Properties $\mathbf{(A.i)}$. To this end we note that if $r_1 =  \maxp{}{i : d^i \leq n}$ then it is sufficient to show that Properties $\mathbf{(A.i)}$ hold for all $r \leq r_1$. It is clear then that the result follows from the following claim, which we again verify at the end of the proof.

\begin{claim}\label{claim:inequalities:proof:gknp:two}
For all $r \leq r_1$, we have
\begin{align}
\frac{1}{2^{37}} d^r \leq \dhat^r \leq 2^{37} d^r.\label{Proof:Prop_estimates:d_inequality}
\end{align}  
\end{claim}
\noindent
This completes the proof.
\end{proof}

\bigskip

It remains to prove Claims \ref{claim:inequalities:proof:gknp} and \ref{claim:inequalities:proof:gknp:two}.

\bigskip

\begin{proof}[Proof of \Cref{claim:inequalities:proof:gknp}]
Note first that as $\epsilon_j \leq \frac{1}{2}$ for all $1 \leq j \leq r$, we can estimate the product as
\begin{equation}\label{eq:Proof_Gknp_Claim:eq1}
\prod_{j=0}^{r}\left(1-\epsilon_j \right) = \prod_{j=1}^{r}\left(1-\epsilon_j \right)\geq \exp\biggl(-2 \sum_{j=1}^{r} \epsilon_j \biggr).  
\end{equation}

By reversing the order of summation we can write this sum as 
\begin{align}
\sum_{j=1}^{r} \epsilon_j &= \sum_{j=1}^{r} \frac{5}{\log n -\log \left(2a\dhat^{j} \right)} 
= 5\sum_{j=1}^{r} \frac{1}{\log n - \log \left(2a\dhat^{r-j+1} \right)}. \label{eq:Proof_Gknp_Claim:eq2}   
\end{align}
To bound the term $\log \left(2a\dhat^{r-j+1} \right)$ from above, we first note that since $a\dhat^r \leq \frac{n}{2\log n}$, we have
\begin{align*}
\log \left(2a\dhat^{r} \right) \leq \log \left(\frac{n}{\log n} \right) = \log n -\log \log n.
\end{align*}
Furthermore, since $\dhat = \omega\left(k \log^3 n\right) = \omega\left( \log^4 n\right)$, it follows that for all $j\in \N$,
\begin{align*}
\log \left(\dhat^{j} \right) \geq 4 j \log \log n,
\end{align*}
which, together with the previous equation, yields that 
\begin{align*}
\log \left(2a\dhat^{r_0-j+1} \right) \leq \log n - \log \log n - 4(j-1) \log \log n.
\end{align*}

Combining this bound with \eqref{eq:Proof_Gknp_Claim:eq2} we obtain:

\begin{align}\label{eq:Proof_Gknp_Claim:eq3}
\sum_{j=1}^{r_0} \epsilon_j \leq 5 \sum_{j=1}^{r_0} \frac{1}{(4j-3) \log \log n} \leq \frac{5}{\log \log n} \left(1+ \sum_{j=1}^{r_0} \frac{1}{ 4j} \right) \leq \frac{5}{4} \cdot \frac{5+\log r_0}{\log \log n} \leq \frac{5}{4},
\end{align} 
where the penultimate inequality comes from a standard bound on the harmonic sum and the last inequality uses $r_0 \leq \frac{\log n}{e^5}$. Plugging \eqref{eq:Proof_Gknp_Claim:eq3} into \eqref{eq:Proof_Gknp_Claim:eq1} we obtain
\[
\prod_{j=0}^{r_0}\left(1-\epsilon_j \right) \geq \exp\left(-\frac{5}{2}\right) \geq \frac{1}{2^4},
\]
which concludes the proof of \eqref{Proof:Prop_estimates:epsilon_inequality}.
\end{proof}

\begin{proof}[Proof of \Cref{claim:inequalities:proof:gknp:two}]
Clearly, it is sufficient to prove the claim for $r=r_1$, and we note that $r_1 \geq 1$, since $\dhat \leq n$. If $r_1$ = $1$, we get from Property \ref{Property:neighbourhood_estimates:vNhd}' 
\begin{align*}
\frac{1}{2^{37}} \dhat \leq |\vNhd[1]{v}| \leq 2^{37} \dhat,
\end{align*} which immediately implies \eqref{Proof:Prop_estimates:d_inequality}. 

Otherwise, note that for each vertex $v \in [n]$ and each $\epsilon \in \left(0,\frac{1}{2}\right]$ such that $\left(\frac{\dhat}{n}\right)^\epsilon \leq 2^{-6}$, we have by Lemma \ref{l:vertex_expansion} that
\[
(1-\epsilon) \left(1-\estimerror \right)\dhat \leq |\vNhd[1]{v}| \leq \left(1+\estimerror\right) \dhat.
\]
Therefore, recalling that $d = d\left(\Gknp\right) = \frac{1}{n} \sum_{v \in [n]} |\vNhd[1]{v}|$, we have 
\begin{equation}\label{e:dbound}
(1-\epsilon) \left(1-\estimerror\right)\dhat \leq d \leq \left(1+\estimerror \right) \dhat.
\end{equation}

As $r_0 \geq 2$, it follows that $\dhat \leq \sqrt{n}$ and so we can take $\epsilon = \estimerror$ and by \eqref{e:dbound} obtain the following:
\[
\frac{1}{2} \dhat^r \leq (1-\epsilon)^r \left(1-\estimerror \right)^r\dhat^r \leq d^r \leq \left(1+\estimerror \right)^r \dhat^r \leq 2 \dhat^r.
\]
\end{proof}

\section{Concluding discussion}\label{sec:Discussion}

In \Cref{thm:main:Meyniel} we prove Conjecture \ref{conj:MeynielHyper} for $\Gknp$ in dense regimes up to a $\log$-factor. It would be interesting to remove this $\log$-factor to match the conjectured bound. We note that in the case of $G(n,p)$, Pra\l at and Wormald \cite{Pralat_Wormald_CopsMeyniels} gave a $2$-stage cop strategy with similarities to the vertex surrounding strategy of \L uczak and Pra\l at \cite{Luczak_Pralat_RobberZigZag} to show that Meyniel's Conjecture holds whp in $\Gnp$ in the dense regime. Similar ideas might be of use in order to remove this $\log$-factor.

Furthermore, our theorems give upper bounds on the cop number. It would be interesting to know if these upper bounds are close to tight. \L uczak and Pra\l at \cite{Luczak_Pralat_RobberZigZag} gave an escape strategy for the Robber which matches their upper bound up to logarithmic factors, which also uses in a critical way the regular (vertex-)expansion properties of $\Gnp$. However, whilst there are already some technical issues to overcome with extending their analysis of such a strategy to the case of $\Gknp$ with $k = \omega(1)$, the robber must also have to take into account in some way the `edge-expansion' of $\Gknp$, as we know there are regimes of $p$ where the cops can do strictly better than the bounds given by the vertex surrounding strategy. Hence some new ideas will be necessary to find a corresponding robber strategy in the hypergraph game.

More generally, it seems the game of Cops and Robber on hypergraphs has been much less well studied than the graphical counterpart. In particular, it would be interesting to find some natural classes of hypergraphs on which the cop number is bounded. Motivated by the classic result of Aigner and Fromme \cite{AiFr1984} on the cop number of planar graphs, it is natural to ask such questions in relation to geometric notions of embeddability.

However, even in the case $k=3$, it is clear that there are $k$-graphs which are embeddable without crossings in $\mathbb{R}^k$ but have arbitrary large cop number. Hence, in order to bound the cop number we need to make further assumptions on the structure of the $k$-graphs, and a natural one in the case of $3$-graphs would be to ask that the $3$-graph, when viewed as a $2$-dimensional simplicial complex, is simply connected.

\begin{question}
Is there a constant $K$ such that every simply connected $3$-graph $G$ which can be embedded without crossings in $\mathbb{R}^3$ satisfies $c(G) \leq K$?
\end{question}

For simply connected $3$-graphs, it is known that embeddability in $\mathbb{R}^3$, as with Kuratowski's theorem, has a characterisation in terms of excluded minors \cite{C17,C172}, here in terms of \emph{space minors}. In the case of graphs, a result of Andreae \cite{A86} shows that excluding a fixed minor bounds the cop number of a graph, and it would be interesting to know if, for an appropriate notion of minor, this also holds for the hypergraph game.

\begin{question}
Let $H$ be a fixed $k$-graph. Does there exist a constant $K:=K(H)$ such that every $k$-graph $G$ with no $H$-`minor' satisfies $c(G) \leq K$?
\end{question}

Finally, we note that there are perhaps other natural ways to extend the game of Cops and Robber to the hypergraph setting.  One particularly natural variant would be to play the game on the \emph{edges} of the hypergraph, rather than the vertices. That is, the cops and robber live on the edges of the hypergraph and are allowed to move to incident edges, instead of adjacent vertices. Let us denote the minimum number of cops needed to win in the \emph{edge-game} on $G$ by $c_e(G)$. In the case of graphs, the edge-game was considered by Dudek, Gordinowicz and Pra\l at \cite{DuGoPr2014}, who showed that it is closely related to the vertex-game, in that for any graph $G$,
\begin{align}
\left\lceil \frac{c(G)}{2} \right\rceil \leq c_e(G) \leq c(G)+1. \label{e:edge}
\end{align}
 In particular, $c(G)$ and $c_e(G)$ cannot differ by more than a fixed multiplicative constant. In the hypergraph game, it is less clear whether $c(G)$ and $c_e(G)$ can have wildly different behaviour.
\begin{question}
Is there a function $f:\mathbb{N} \rightarrow \mathbb{R}$ such that for any $k$-graph $G$,
\[
\frac{1}{f(k)} c(G) \leq c_e(G) \leq f(k) c(G)?
\]
\end{question}

\section*{Acknowledgements}
This work is supported in part by the Austrian Science Fund (FWF) [10.55776/\{W1230, P36131, I6502\}]. The fourth author is supported by the NSERC Discovery Grant R611450 (Canada). This work is the result of the fourth author's visit of TU Graz under Oberwolfach's Simons Visiting Professors program in December 2021. For the purpose of open access, the authors have applied a CC BY public copyright licence to any Author Accepted Manuscript version arising from this submission. 
\bibliography{references}
\bibliographystyle{abbrv}

\end{document}